%% file: nice_thermodynamics.tex
\documentclass[11pt]{amsart}
\usepackage{amsmath}
\usepackage{amssymb}
\usepackage{amsfonts}
\usepackage{mathrsfs}
\usepackage{epsfig}
\usepackage{color}

\DeclareMathAlphabet{\mathpzc}{OT1}{pzc}{m}{it}

\numberwithin{equation}{section}       

\theoremstyle{plain}

\newtheorem*{theo }{Theorem}
\newtheorem{prop}{Proposition}[section]

\newtheorem{coro}[prop]{Corollary}
\newtheorem{lemm}[prop]{Lemma}
\newtheorem{theoalph}{Theorem}

\theoremstyle{definition}
\newtheorem{defi}[prop]{Definition}

\theoremstyle{remark}
\newtheorem{rema}[prop]{Remark}

\newtheoremstyle{citing}
  {3pt}
  {3pt}
  {\itshape}
  {}
  {\bfseries}
  {.}
  {.5em}
  {\thmnote{#3}}

\theoremstyle{citing}
\newtheorem*{generic}{}

\newcommand{\partn}[1]{{\smallskip \noindent \textbf{#1.}}}

\newcommand{\C}{\mathbb{C}}

\newcommand{\N}{\mathbb{N}}

\newcommand{\R}{\mathbb{R}}

\newcommand{\cO}{\mathcal{O}}

\newcommand{\cR}{\mathcal{R}}
\newcommand{\cS}{\mathcal{S}}

\newcommand{\fD}{\mathfrak{D}}

\newcommand{\fL}{\mathfrak{L}}

\newcommand{\sC}{\mathscr{C}}

\newcommand{\sE}{\mathscr{E}}

\newcommand{\sM}{\mathscr{M}}
\newcommand{\sN}{\mathscr{N}}

\newcommand{\sR}{\mathscr{R}}

\newcommand{\sW}{\mathscr{W}}

\newcommand{\hQ}{\widehat{Q}}

\newcommand{\hV}{\widehat{V}}
\newcommand{\hW}{\widehat{W}}

\newcommand{\hmu}{\widehat{\mu}}

\newcommand{\hrho}{\widehat{\rho}}

\newcommand{\tJ}{\widetilde{J}}

\newcommand{\tV}{\widetilde{V}}

\newcommand{\tgamma}{\widetilde{\gamma}}

\newcommand{\teta}{\widetilde{\teta}}

\newcommand{\trho}{\widetilde{\rho}}

\newcommand{\tsigma}{\widetilde{\tsigma}}
\newcommand{\tvarsigma}{\widetilde{\tvarsigma}}

\newcommand{\ov}{\overline}

\renewcommand{\=}{ : = }

\DeclareMathOperator{\diam}{diam}

\DeclareMathOperator{\dist}{dist}

\DeclareMathOperator{\HD}{HD}
\DeclareMathOperator{\Crit}{Crit}
\DeclareMathOperator{\CV}{\operatorname{CV}}

\newcommand{\CC}{\overline{\C}}
\newcommand{\RR}{\overline{\R}}

\newcommand{\map}{f} 
\newcommand{\CJ}{\Crit(\map) \cap J(\map)} 
\newcommand{\CVJ}{\CV(\map) \cap J(\map)} 
\newcommand{\HDhyp}{\HD_{\operatorname{hyp}}}

\DeclareMathOperator{\hyp}{hyp}

\DeclareMathOperator{\con}{con}
\newcommand{\Jcon}{J_{\con}}

\newcommand{\TCE}{Topological Collet-Eckmann Condition}
\newcommand{\TCEC}{TCE condition}

\newcommand{\tpos}{t_+} 
\newcommand{\tneg}{t_-} 
\newcommand{\chiinf}{\chi_{\inf}}
\newcommand{\chisup}{\chi_{\sup}} 
\newcommand{\infd}{\chiinf^*}
\newcommand{\supd}{\chisup^*}
\newcommand{\pressure}{\mathscr{P}}

\newcommand{\badp}{Y}
\newcommand{\uW}{\underline{W}}
\newcommand{\mconst}{m} 

\newcommand{\puno}{(i)}
\newcommand{\pdos}{(ii)}

\begin{document}

\title[Thermodynamics of rational maps]{Nice inducing schemes and the \\ thermodynamics of rational maps}
\author[F. Przytycki]{Feliks Przytycki$^\dag$}
\author[J. Rivera-Letelier]{Juan Rivera-Letelier$^\ddag$}
\thanks{$\dag$ Partially supported by Polish MNiSW Grant NN201 0222 33 and the EU FP6 Marie Curie ToK and RTN programmes SPADE2 and CODY}
\thanks{$\ddag$ Partially supported by Research Network on Low Dimensional Dynamical Systems, PBCT/CONICYT, Chile, Swiss National Science Foundation Projects No. 200021-107588 and 200020-109175, and IMPAN}
\address{$\dag$ Feliks Przytycki, Institute of Mathematics, Polish Academy of Sciences, ul. \'Sniadeckich 8, 00956 Warszawa, Poland.}
\email{feliksp@impan.gov.pl}
\address{$\ddag$ Juan Rivera-Letelier, Facultad de Matem{\'a}ticas, Campus San Joaqu{\'\i}n, P. Universidad Cat{\'o}lica de Chile, Avenida Vicu{\~n}a Mackenna~4860, Santiago, Chile}
\email{riveraletelier@mat.puc.cl}
\begin{abstract}
We study the thermodynamic formalism of a complex rational map~$f$ of degree at least two, viewed as a dynamical system acting on the Riemann sphere.
More precisely, for a real parameter~$t$ we study the existence of equilibrium states of~$f$ for the potential $-t \ln |f'|$, and the analytic dependence on~$t$ of the corresponding pressure function.
We give a fairly complete description of the thermodynamic formalism for a large class of rational maps, including well known classes of non\nobreakdash-uniformly hyperbolic rational maps, such as (topological) Collet-Eckmann maps, and much beyond.
In fact, our results apply to all non\nobreakdash-renormalizable polynomials without indifferent periodic points, to infinitely renormalizable quadratic polynomials with \emph{a priori} bounds, and all quadratic polynomials with real coefficients.
As an application, for these maps we describe the dimension spectrum for Lyapunov exponents, and for pointwise dimensions of the measure of maximal entropy, and obtain some level-1 large deviations results.
For polynomials as above, we conclude that the integral means spectrum of the basin of attraction of infinity is real analytic at each parameter in~$\R$, with at most two exceptions.
\end{abstract}

\maketitle
\setcounter{tocdepth}{1}
\tableofcontents

\section{Introduction}
The purpose of this paper is to study the thermodynamic formalism of a complex rational map~$f$ of degree at least two, viewed as a dynamical system acting on the Riemann sphere $\overline{\mathbb{C}}$.
More precisely, for a real parameter~$t$ we study the existence of equilibrium states of~$f$ for the potential $-t \ln |f'|$ and the (real) analytic dependence on~$t$ of the corresponding pressure function.
Our particular choice of potentials is motivated by the close connection between the corresponding pressure function and various multifractal spectra.
In fact, we give applications of our results to rigidity, multifractal analysis of dimension spectrum for Lyapunov exponents and for pointwise dimensions, as well as level\nobreakdash-1 large deviations.
See~\cite{BakSta96,BelSmi05,Ere91} for other applications of the thermodynamic formalism of rational maps to complex analysis.

For $t < 0$ and for an arbitrary rational map~$f$, a complete description of the thermodynamic formalism was given by Makarov and Smirnov in~\cite{MakSmi00}.
They showed that the corresponding transfer operator is quasi-compact on a suitable Sobolev space, see also~\cite{Rue92}.
For~$t = 0$ and a general rational map~$f$, there is a unique equilibrium state of~$f$ for the constant potential equal to~$0$~\cite{Lju83,FreLopMan83}.
To the best of our knowledge it is not known if for a general rational map~$f$ the pressure function is real analytic on a neighborhood of~$t = 0$.
For $t > 0$ the only results on the analyticity of the pressure function that we are aware of, are for generalized polynomial-like maps without recurrent critical points in the Julia set.
For such a map the analyticity properties of the pressure function were studied in~\cite{MakSmi03,StrUrb03}, using a Markov tower extension and an inducing scheme, respectively.

Under very weak hypotheses on a rational map~$f$, we show that the pressure function is real analytic at each parameter~$t$ in~$\R$, with at most two exceptions.
In other words, the pressure function can have at most two phase transitions and thus at most three phases.
It turns out that the parameter $t = 0$ is always contained in one of the phases, which is characterized as the only phase where the measure theoretic entropy of an equilibrium state can be strictly positive.
We show that for every parameter in this phase there is a unique equilibrium state that has exponential decay of correlations and that satisfies the Central Limit Theorem.

Our results apply to well-known classes of non\nobreakdash-uniformly hyperbolic rational maps.
Furthermore our results apply to all non\nobreakdash-renormalizable polynomials without indifferent periodic points, to infinitely renormalizable quadratic polynomials with \emph{a priori} bounds, and to all quadratic polynomials with real coefficients.

The main ingredients in our approach are the distinct characterizations of the pressure function given in~\cite{PrzRivSmi04} and the inducing scheme introduced in~\cite{PrzRiv07}, which we develop here in a more general setting.
It is worth noticing that to study a rational map with a recurrent critical point in the Julia set, it is usually not enough to consider an induced map defined with the first return time.
The induced maps considered here are constructed with higher returns times, which makes the estimates more delicate.
As in~\cite{PrzRiv07}, our key estimates are based on controlling a discrete version of conformal mass.
However, the ``density'' introduced in~\cite{PrzRiv07} for this purpose does not work in the more general setting considered here.
We thus introduce a different technique, based on a Whitney type decomposition.

There have been several recent results on the thermodynamic formalism of multimodal interval maps with non-flat critical points, by Bruin and Todd~\cite{BruTod08,BruTod09a} and Pesin and Senti~\cite{PesSen08}.
Besides~\cite[Theorem~6]{BruTod08}, that gives a complete description of the thermodynamic formalism for~$t$ close to~0 and for a general topologically transitive multimodal interval map with non-flat critical points, all the results that we are aware of are restricted to non\nobreakdash-uniformly hyperbolic maps.
It is possible to apply the approach given here to obtain a fairly complete description of the thermodynamic formalism of a general topologically transitive multimodal interval map with non-flat critical points.
We obtain in particular that the pressure function of such a map is real analytic at each parameter in~$\mathbb{R}$, with at most two exceptions.\footnote{Recently Iommi and Todd~\cite{IomTod0907} have shown similar results for transitive multimodal maps with non-flat critical points as those presented here, but only obtaining that the pressure function is continuous differentiable, and without statistical properties of the equilibrium states.}
We are in the process of writing these results.

After reviewing some general properties of the pressure function in~\S\ref{ss:pressure and equilibria}, we state our main result in~\S\ref{ss:nice thermodynamics}.
The applications to rigidity, multifractal analysis, and level\nobreakdash-1 large deviations are given in Appendix~\ref{s:applications}.

Throughout the rest of this introduction we fix a rational map~$f$ of degree at least two, we denote by~$\Crit(f)$ the set of critical points of~$f$ and by~$J(f)$ the Julia set of~$f$.
\subsection{The pressure function and equilibrium states}\label{ss:pressure and equilibria}
We give here the definition of the pressure function and of equilibrium states, see~\S\ref{s:preliminaries} for references and precise formulations.

Let $\sM(f)$ be the space of all probability measures supported on~$J(f)$ that are invariant by~$f$.
We endow~$\sM(f)$ with the weak$^*$ topology.
For each $\mu \in \mathscr{M}(f)$, denote by~$h_\mu(f)$ the \emph{measure theoretic entropy of}~$\mu$, and by $\chi_\mu(f) \= \int \ln |f'| d \mu$ the \emph{Lyapunov exponent of}~$\mu$.
Given a real number~$t$ we define the \emph{pressure of $f|_{J(f)}$ for the potential $- t \ln |f'|$} by,
\begin{equation}\label{e:variational principle}
P(t) \=
\sup \left\{ h_\mu(f) - t \chi_\mu(f) \mid \mu \in \mathscr{M}(f) \right\}.
\end{equation}
For each $t \in \R$ we have $P(t) < + \infty$,\footnote{When $t \le 0$ the number $P(t)$ coincides with the topological pressure of~$f|_{J(f)}$ for the potential $- t \ln |f'|$, defined with $(n, \varepsilon)$\nobreakdash-separated sets.
However, these numbers do not coincide when $t > 0$ and there are critical points of~$f$ in~$J(f)$.
In fact, since $\ln |f'|$ takes the value~$- \infty$ at each critical point of~$f$, in this case the topological pressure of~$f|_{J(f)}$ for the potential $- t \ln |f'|$ is equal to~$+\infty$.}
and the function $P : \R \to \R$ so defined will be called \emph{the pressure function} of~$f$.
It is convex, non\nobreakdash-increasing and Lipschitz continuous.

An invariant probability measure~$\mu$ supported on the Julia set of~$f$ is called an \emph{equilibrium state of~$f$ for the potential $-t\ln|f'|$}, if the supremum~\eqref{e:variational principle} is attained for this measure.

The numbers,
$$
\chiinf(f) \= \inf \left\{ \chi_\mu(f) \mid \mu \in \mathscr{M}(f) \right\},
$$
$$
\chisup(f) \= \sup \left\{ \chi_\mu(f) \mid \mu \in \mathscr{M}(f) \right\},
$$
will be important in what follows.
We call
\begin{equation}\label{e:negative phase transition}
\tneg \= \inf \{ t \in \R \mid P(t) + t \chisup(f) > 0 \}
\end{equation}
\begin{equation}\label{e:positive phase transition}
\tpos \= \sup \{ t \in \R \mid P(t) + t \chiinf(f) > 0 \}
\end{equation}
the \emph{condensation point} and the \emph{freezing point} of~$f$, respectively.
We remark that the condensation (resp. freezing) point can take the value $-\infty$ (resp. $+ \infty$).
We have the following properties (Proposition~\ref{p:asymptotes}):
\begin{itemize}
\item
  $\tneg < 0 < \tpos$;
\item
for all $t \in \R \setminus (\tneg, \tpos)$ we have $P(t) = \max \{ - t \chisup(f), - t \chiinf(f) \} $;
\item
for all $t \in (\tneg, \tpos)$ we have $P(t) > \max \{ - t \chiinf(f), - t \chisup(f) \}$.
\end{itemize}

\subsection{Nice sets and the thermodynamics of rational maps}\label{ss:nice thermodynamics}
A neighborhood~$V$ of $\CJ$ is a \emph{nice set for}~$f$, if for every $n \ge 1$ we have $f^n(\partial V) \cap V = \emptyset$, and if each connected component of~$V$ is simply connected and contains precisely one critical point of~$f$ in~$J(f)$.
A \emph{nice couple for}~$f$ is a pair of nice sets~$(\hV, V)$ for~$f$ such that $\ov{V} \subset \hV$ and such that for every $n
\ge 1$ we have $f^n(\partial V) \cap \hV = \emptyset$.
We will say that a nice couple $(\hV, V)$ is \emph{small}, if there is a small $r > 0$ such that $\hV \subset B(\CJ, r)$.

We say that a rational map~$f$ is \emph{expanding away from critical points}, if for every neighborhood~$V'$ of~$\CJ$ the map~$f$ is uniformly expanding on the set
$$
\{ z \in J(f) \mid \text{ for every $n \ge 0$, $f^n(z) \not \in V'$} \}.
$$

\begin{generic}[Main Theorem]
Let~$f$ be a rational map of degree at least two that is expanding away from critical points, and that has arbitrarily small nice couples.
Then following properties hold.
\begin{description}
\item[Analyticity of the pressure function]
The pressure function of~$f$ is real analytic on $(\tneg, \tpos)$, and linear with slope~$- \chisup(f)$ (resp. $-\chiinf(f)$) on $(- \infty, \tneg]$ (resp. $[\tpos, + \infty)$).
\item[Equilibrium states]
For each $t_0 \in (\tneg, \tpos)$ there is a unique equilibrium state of~$f$ for the potential~$-t_0 \ln|f'|$.
Furthermore this measure is ergodic and mixing.
\end{description}
\end{generic}
We now list some classes of rational maps for which the Main Theorem applies.
\begin{itemize}
\item
Using~\cite{KozvSt09} we show that each \emph{at most finitely renormalizable polynomial without indifferent periodic orbits} satisfies the hypotheses of the Main Theorem, see Theorem~\ref{t:non-renormalizable} in~\S\ref{ss:non-renormalizable}.
\item
\emph{Quadratic polynomials with real coefficients} satisfy the hypothesis of the Main Theorem, with two exceptions: Maps with an indifferent periodic point, which are considerably simpler to treat, and maps having a renormalization conjugated to the Feigenbaum polynomial, for which we show that a slightly more general version of the Main Theorem applies (Theorem~\ref{t:nice thermodynamics} in~\S\ref{s:proof of nice thermodynamics}).
In particular our results imply that the conclusions of the Main Theorem hold for each quadratic polynomial with real coefficients, see~\S\ref{ss:real quadratic} for details.
\item
\emph{Topological Collet-Eckmann rational maps} have arbitrarily small nice couples~\cite[Theorem~E]{PrzRiv07} and are expanding away of critical points.
These maps include \emph{Collet-Eckmann rational maps}, as well as maps without recurrent critical points and without parabolic periodic points; see~\cite{PrzRoh98} and also~\cite[Main Theorem]{PrzRivSmi03}.
\item
Each \emph{backward contracting rational map} has arbitrarily small nice couples~\cite[Proposition~6.6]{Riv07}.
If in addition the Julia set is different from~$\CC$, such a map is also expanding away from critical points~\cite[Corollary~8.3]{Riv07}.
In~\cite[Theorem~A]{Riv07} it is shown that a rational map~$f$ of degree at least two satisfying the \emph{summability condition with exponent~1}:
\begin{quote}
$f$ does not have indifferent periodic points and for each critical value~$v$ in the Julia set of~$f$ we have
$$ \sum_{n = 1}^{+ \infty} |(f^n)'(v)|^{-1} < + \infty $$
\end{quote}
is backward contracting, and it thus has arbitrarily small nice couples.
In~\cite{Prz98} it is shown that each rational map satisfying the summability condition with exponent~1 is expanding away of critical points.
\end{itemize}

Using a stronger version of the Main Theorem (Theorem~\ref{t:nice thermodynamics} in~\S\ref{s:proof of nice thermodynamics}), we show that each infinitely renormalizable quadratic polynomial for which the diameters of the small Julia sets converge to~$0$ satisfies the conclusions of the Main Theorem, see~\S\ref{ss:infinitely renormalizable} in Appendix~\ref{s:nice puzzles}.
In particular the conclusions of the Main Theorem hold for each infinitely renormalizable polynomial with \emph{a priori} bounds; see~\cite{KahLyu08,McM94} and references therein for results on~\emph{a priori} bounds.
\begin{rema}
In the proof of the Main Theorem we construct the equilibrium states through an inducing scheme with an exponential tail estimate, that satisfies some additional technical properties; see~\S\ref{ss:equilibrium} for precise statements.
The results of~\cite{You99} imply that the equilibrium states in the Main Theorem are exponentially mixing and that the Central Limit Theorem holds for these measures.
It also follows that these equilibrium states have other statistical properties, such as the ``almost sure invariant principle'', see e.g.~\cite{Gou05,MelNic05,MelNic08,Tyr05}.
\end{rema}

We obtain as a direct consequence of the Main Theorem the following result on the integral means spectrum.
\begin{coro}
Let~$f$ be a monic polynomial with connected Julia set and degree~$d \ge 2$, that is expanding away from critical points and that has arbitrarily small nice couples.
Let
$$ \phi : \{ z \in \C \mid |z| > 1 \} \to \C \setminus J(f) $$
be a conformal representation that is tangent to the identity at infinity.
Then the integral means spectrum of~$\phi$,
$$ \beta_{\phi}(t)
\=
\limsup_{r \to 1^+}
\frac{ \ln \int_{0}^{2 \pi} | \phi'(r \exp(i \theta)) |^t d \theta}{| \ln (r - 1) |}, $$
is real analytic on~$(\tneg, \tpos)$ and linear with slope~$1 - \chisup(f) / \ln d$ (resp. $1 -\chiinf(f)/\ln d$) on $(- \infty, \tneg]$ (resp. $[\tpos, + \infty)$).
\end{coro}
This corollary follows directly from the fact that for each~$t \in \R$ we have~$\beta_\phi(t) = P(t) / \ln d + t - 1$, see for example~\cite[Lemma~2]{BinMakSmi03}.

We will now consider several known results related to~the Main Theorem.

As mentioned above, Makarov and Smirnov showed in~\cite{MakSmi00} that the conclusions of the Main Theorem hold for every rational map on $(- \infty, 0)$.
Furthermore, they characterized all those rational maps whose condensation point~$\tneg$ is finite; see~\S\ref{ss:Lyapunov spectrum}.

For a uniformly hyperbolic rational map we have $\tneg = - \infty$ and $\tpos = + \infty$, and for a sub-hyperbolic polynomial with connected Julia set we have $\tpos = + \infty$~\cite{MakSmi96}.
The freezing point~$\tpos$ is finite whenever~$f$ does not satisfy the \TCE\footnote{By~\cite[Main Theorem]{PrzRivSmi03}~$f$ satisfies the \TCE{} if, and only if, $\chiinf(f) > 0$.} (Proposition~\ref{p:asymptotes}).
In fact, in this case the freezing point~$\tpos$ is the first zero of the pressure function.
On the other hand, there is an example in~\cite[\S3.4]{MakSmi03} of a generalized polynomial-like map satisfying the \TCE\footnote{In fact this map has the stronger property that no critical point in its Julia set is recurrent.} and whose freezing point~$\tpos$ is finite.

When~$f$ is a generalized polynomial-like map without recurrent critical points, the part of the Main Theorem concerning the analyticity of the pressure function was shown in~\cite{MakSmi00,MakSmi03,StrUrb03}.
Note that the results of~\cite{StrUrb03} apply to maps with parabolic periodic points.

Let us also mention that, if~$f$ is an at most finitely renormalizable polynomial without indifferent periodic points and such that for every critical value~$v$ in~$J(f)$
$$ \lim_{n \to + \infty} |(f^n)'(v)| = + \infty, $$
and if~$t_0 > 0$ is the first zero of the pressure function, then the absolutely continuous invariant measure constructed in~\cite{RivShe1004} is an equilibrium state of~$f$ for the potential~$-t_0 \ln |f'|$, see also~\cite{GraSmi09,PrzRiv07}.

In the case of a general transitive multimodal interval map with non-flat critical points, a result analogous to the Main Theorem was shown by Bruin and Todd in~\cite[Theorem~6]{BruTod08} for~$t$ in a neighborhood of~$0$.
Similar results for~$t$ in a neighborhood of~$[0, 1]$ were shown by Pesin and Senti in~\cite{PesSen08} for multimodal interval maps with non-flat critical points satisfying the Collet-Eckmann condition and some additional properties (see also \cite[Theorem~2]{BruTod09a}) and by Bruin and Todd in~\cite[Theorem~1]{BruTod09a}, for~$t$ in a one-sided neighborhood of~$1$, and for multimodal interval maps with non-flat critical points and with a polynomial growth of the derivatives along the critical orbits; see also~\cite{BruKel98}.

In~\cite[Proposition~7]{Dob09}, Dobbs shows that there is a quadratic polynomial with real coefficients~$f_0$ such that the pressure function, \emph{defined for the restriction of~$f_0$ to a certain compact interval}, has infinitely many phase transitions before it vanishes.
This behavior of~$f_0$ as an interval map is in sharp contrast with its behavior as a complex map: Our results imply that the pressure function of~$f_0$, viewed as a map acting on the (complex) Julia set of~$f_0$, is real analytic before it vanishes.

\subsection{Notes and references}\label{ss:notes and references}
See the book~\cite{Rue04} for an introduction to the thermodynamic formalism and~\cite{PUbook,Zin96} for an introduction in the case of rational maps.

For results concerning other potentials, see~\cite{DenUrb91e,GelWol07,Prz90,Urb03c} for the case of rational maps, and~\cite{BruTod08,PesSen08} and references therein for the case of multimodal interval maps with non-flat critical points.

For a rational map~$f$ satisfying the \TCE{} and for~$t = \HDhyp(f)$, the construction of the corresponding equilibrium state given here gives a new proof of the existence of an absolutely continuous invariant measure, with respect to a conformal measure.
More precisely, it gives a new proof of~\cite[Key Lemma]{PrzRiv07}.

\subsection{Strategy and organization}
We now describe the strategy of the proof of the Main Theorem, and simultaneously describe the organization of the paper.
Our results are either well-known or vacuous for rational maps without critical points in the Julia set, so we will (implicitly) assume that all the rational maps we consider have at least one critical point in the Julia set.

In~\S\ref{s:preliminaries} we review some general results concerning the pressure function, including some of the different characterizations of the pressure function given in~\cite{PrzRivSmi04}.
We also review some results concerning the asymptotic behavior of the derivative of the iterates of a rational map.
These results are mainly taken or deduced from results in~\cite{Prz99,PrzRivSmi03,PrzRivSmi04}.

To prove the Main Theorem we make use of the inducing scheme introduced in~\cite{PrzRiv07}, which is developed in the more general setting considered here in~\S\S\ref{s:nice induced maps},~\ref{s:lifting}.
In~\S\ref{ss:nice sets and couples} we recall the definitions of nice sets and couples, and introduce a weaker notion of nice couples that we call ``pleasant couples''.
Pleasant couples will allow us to handle non-primitive renormalizations, see~Remark~\ref{r:primitive}.
Then we recall in~\S\ref{ss:canonical induced map} the definition of the canonical induced map associated to a nice (or pleasant) couple.
We also review the decomposition of its domain of definition into ``first return'' and ``bad pull-backs'' as well as the sub-exponential estimate on the number of bad pull-backs of a given order (\S\ref{ss:bad pull-backs}).
In~\S\ref{ss:two variable pressure} we consider a two variable pressure function associated to such an induced map, that will be very important for the rest of the paper.
This pressure function is analogous to the one introduced by Stratmann and Urbanski in~\cite{StrUrb03}.

In~\S\ref{s:lifting} we give sufficient conditions on a nice (or pleasant) couple so that the conclusions of the Main Theorem hold for values of~$t$ in a neighborhood of an arbitrary $t_0 \in (\tneg, \tpos)$ (Theorem~\ref{t:lifting}).
These conditions are formulated in terms of the two variable pressure function defined in~\S\ref{ss:two variable pressure}.
We follow the method of~\cite{PrzRiv07} for the construction of the conformal measures and the equilibrium states, which is based on the results of Mauldin and Urbanski in~\cite{MauUrb03}.
As in~\cite{PesSen08}, we use a result of Zweim\"uller in~\cite{Zwe05} to show that the invariant measure we construct is in fact an equilibrium state.
The uniqueness is a direct consequence of the results of Dobbs in~\cite{Dob0804}, generalizing~\cite{Led84}.
Finally, we use the method introduced by Stratmann and Urbanski in~\cite{StrUrb03} to show that the pressure function is real analytic.
Here we make use of the fact that the two variable pressure function is real analytic on the interior of the set where it is finite, a result shown by Mauldin and Urbanski in~\cite{MauUrb03}.

The proof the Main Theorem is contained in \S\S\ref{s:Whitney decomposition}, \ref{s:pull-back contribution}, \ref{s:proof of nice thermodynamics}.
The proof is divided into two parts.
The first, and by far the most difficult one, is to show that for $t_0 \in (\tneg, \tpos)$ the two variable pressure associated to a sufficiently small nice (or pleasant) couple is finite on a neighborhood of~$(t, p) = (t_0, P(t_0))$.
To do this we use the strategy of~\cite{PrzRiv07}: we use the decomposition of the domain of definition of the induced map associated to a nice (or pleasant) couple, into first return and bad pull-backs evoked in~\S\ref{ss:bad pull-backs}.
Unfortunately, for values of~$t$ such that $P(t) < 0$, there does not seem to be a natural way to adapt the density introduced in~\cite{PrzRiv07} to estimate the contribution of a bad pull-back.
Instead we use a different argument involving a Whitney type decomposition of a pull-back, which is one of the main technical tools introduced in this paper.
Roughly speaking, we have replaced the ``annuli argument'' of~\cite[Lemma~5.4]{PrzRiv07} by an argument involving ``Whitney squares'', that allow us to make a direct estimate avoiding an induction on the number visits to the critical point.
The Whitney type decomposition is introduced in~\S\ref{s:Whitney decomposition} and the estimate on the contribution of a (bad) pull-back is given in~\S\ref{s:pull-back contribution}.
The finiteness of the two variable pressure function is shown in~\S\ref{ss:finiteness}.
The second part of the proof, that for each~$t$ close to~$t_0$ the two variable pressure function vanishes at~$(t, p) = (t, P(t))$, is given in~\S\ref{ss:vanishing}.
Here we have replaced the analogous (co\nobreakdash-)dimension argument of~\cite{PrzRiv07}, with an argument involving the pressure function of the rational map.

Appendix~\ref{s:nice puzzles} is devoted to show that the conclusions of the Main Theorem hold for several classes of polynomials.
In~\S\ref{ss:non-renormalizable} we show that each at most finitely renormalizable polynomial without indifferent periodic points satisfies the hypotheses of the Main Theorem (Theorem~\ref{t:non-renormalizable}).
Then in~\S\ref{ss:infinitely renormalizable} we show that each infinitely renormalizable quadratic polynomial for which the diameters of small Julia sets converge to~$0$ satisfies the hypotheses of Theorem~\ref{t:nice thermodynamics}.
Finally, \S\ref{ss:real quadratic} is devoted to the case of quadratic polynomials with real coefficients.

In Appendix~\ref{s:applications} we give applications of our main results to rigidity, multifractal analysis, and level\nobreakdash-1 large deviations.
\subsection{Acknowledgments} 
We are grateful to Weixiao Shen and Daniel Smania for their help with references, Weixiao Shen again and Genadi Levin for their help with the non\nobreakdash-renormalizable case and Henri Comman for his help with the large deviations results.
We also thank Neil Dobbs, Godofredo Iommi, Jan Kiwi and Mariusz Urbanski for useful conversations and comments.
Finally, we are grateful to Krzysztof Baranski for making Figure~\ref{f:bad} and the referee for his suggestions and comments that help to clarify some of the concepts introduced in the paper.
\section{Preliminaries}\label{s:preliminaries}
The purpose of this section is to give some some general properties of the pressure function (\S\S\ref{ss:pressure properties}, \ref{ss:conformal measures}), and some characterizations of~$\chiinf$ and~$\chisup$ (\S\ref{ss:individual pressure}).
These results are mainly taken or deduced from the results in~\cite{Prz99,PrzRivSmi03,PrzRivSmi04}.
We also fix some notation and terminology in \S\ref{ss:notation and terminology}, that will be used in the rest of the paper.

Throughout the rest of this section we fix a rational map~$f$ of degree at least two.
We will denote $h_\mu(f), \chi_\mu(f), \ldots$ just by $h_\mu, \chi_\mu, \ldots$ .
For simplicity we will assume that no critical point of~$f$ in the Julia set is mapped to another critical point under forward iteration.
The general case can be handled by treating whole blocks of critical points as a single critical point; that is, if the critical points $c_0, \ldots, c_k \in J(f)$ are such that~$c_i$ is mapped to~$c_{i + 1}$ by forward iteration, and maximal with this property, then we treat this block of critical points as a single critical point. 

\subsection{Notation and terminology}\label{ss:notation and terminology}
We will denote the extended real line by $\RR \= \R \cup \{ - \infty, + \infty \}$.

Distances, balls, diameters and derivatives are all taken with respect to the spherical metric.
For $z \in \CC$ and $r > 0$, we denote by $B(z, r) \subset \CC$ the ball centered at~$z$ and with radius~$r$.

For a given $z \in \CC$ we denote by~$\deg_f(z)$ the local degree of~$f$ at~$z$, and for $V \subset \CC$ and $n \ge 0$, each connected component of $f^{-n}(V)$ will be called \emph{a pull-back of~$V$ by~$f^n$}.
When~$V$ is clear from the context, for such a set~$W$ we put $m_W = n$.
When $n = 0$ we obtain that each connected component~$W$ of~$V$ is a pull-back of~$V$ with $m_W = 0$.
In the case where~$f^n$ is univalent on~$W$ we will say that~$W$ is an \emph{univalent pull-back of~$V$ by~$f^n$}.
Note that the set~$V$ is not assumed to be connected.

We will abbreviate ``Topological Collet-Eckmann'' by TCE.
\subsection{General properties of the pressure function}\label{ss:pressure properties}
Given an integer~$n \ge 1$ let $\Lambda_n : \CC \times \R \to \RR$ be the function defined by
$$
\Lambda_n(z_0, t) \= \sum_{w \in f^{-n}(z_0)} |(f^n)'(z_0)|^{-t}.
$$
Then for every $t \in \R$ and every~$z_0$ in~$\CC$ outside a set of Hausdorff dimension~$0$, we have
\begin{equation}\label{e:tree pressure}
\limsup_{n \to + \infty} \tfrac{1}{n} \ln \Lambda_n(z_0, t) = P(t),
\end{equation}
see~\cite{Prz99,PrzRivSmi04}.

In the following proposition,
\begin{multline*}
\HDhyp(f)
\=
\sup \{ \HD(X) \mid X \text{ compact and invariant subset of~$\CC$} \\
\text{where~$f$ is uniformly expanding} \}.
\end{multline*}
\begin{prop}\label{p:asymptotes}
Given a rational map~$f$ of degree at least two, the function
$$ t \mapsto P(t) + t \chiinf \text{ (resp. } t \mapsto P(t) + t \chisup \text{)}, $$
is convex, non\nobreakdash-increasing, and non\nobreakdash-negative on $[0, + \infty)$ (resp. $(- \infty, 0]$).
Moreover $\tneg < 0$, and we have $\tpos \ge \HDhyp(f)$ with strict inequality if, and only if,~$f$ satisfies the \TCEC.

In particular for all~$t$ in $(\tneg, \tpos)$ we have $ P(t) >  \max \{ - t \chiinf, - t \chisup \}, $ and for all~$t$ in $\R \setminus (\tneg, \tpos)$ we have $P(t) = \max \{ - t \chiinf, - t \chisup \}$.
\end{prop}
\begin{proof}
For each ~$\mu \in \sM(f)$ the function $t \mapsto h_\mu(f) - t(\chi_\mu - \chiinf)$ (resp. $t \mapsto h_\mu - t(\chi_\mu - \chisup)$) is affine and non\nobreakdash-increasing on $[0, + \infty)$ (resp. $(-\infty, 0]$).
As by definition
$$ P(t) = \sup \{ h_\mu - t \chi_\mu \mid \mu \in \sM(f) \}, $$
we conclude that the function $t \mapsto P(t) + t \chiinf$ (resp. $t \mapsto P(t) + t \chiinf$) is convex and non\nobreakdash-increasing on $[0, + \infty)$ (resp. $(-\infty, 0]$).
It also follows from the definition that $t \mapsto P(t) + t \chiinf$ (resp. $t \mapsto P(t) + t \chiinf$) is non\nobreakdash-negative on this set.

The inequalities~$\tneg < 0$ and $\tpos \ge \HDhyp(f)$ follow from the fact that $\chiinf$ is non\nobreakdash-negative and from the fact that the pressure function~$P$ is strictly positive on $(0, \HDhyp(f))$~\cite{Prz99}.
When~$f$ satisfies the \TCEC, then $\chiinf > 0$~\cite[Main Theorem]{PrzRivSmi03} and thus $\tpos > \HDhyp(f)$.
When $f$ does not satisfy the \TCEC, then $\chiinf = 0$~\cite[Main Theorem]{PrzRivSmi03} and therefore the equality $\tpos = \HDhyp(f)$ follows from the fact that $\HDhyp(f)$ is the first zero of the function~$P$~\cite{Prz99}.
\end{proof}

\subsection{The pressure function and conformal measures}\label{ss:conformal measures}
For real numbers~$t$ and~$p$ we will say that a finite Borel measure~$\mu$ is $(t, p)$-\emph{conformal} for~$f$, if for each Borel subset~$U$ of~$\CC$ on which~$f$ is injective we have
$$ \mu(f(U)) = \exp(p) \int_U |f'|^t d \mu. $$
By the locally eventually onto property of~$f$ on~$J(f)$ it follows that if the topological support of a $(t, p)$-conformal measure is contained in~$J(f)$, then it is in fact equal to~$J(f)$.
\begin{prop}\label{p:conformal measures}
Let~$f$ be a rational map of degree at least two.
Then for each $t \in (\tneg, + \infty)$ there exists a $(t, P(t))$-conformal measure for~$f$ supported on $J(f)$, and for each real number~$p$ for which there is a $(t, p)$-conformal measure for~$f$ supported on~$J(f)$ we have $p \ge P(t)$.
\end{prop}
\begin{proof}
When~$t = 0$, the assertions are well known, see for example~\cite[p.~104]{DenUrb91e}.
The case~$t > 0$ is given by~\cite[Theorem~A]{PrzRivSmi04}.
In the case $t \in (\tneg, 0)$ the existence is given by~\cite[\S3.5]{MakSmi00} (see also~\cite[Theorem~A.7]{PrzRivSmi04}), and in~\cite[Proposition~A.11]{PrzRivSmi04} it is shown that if for some real number~$p$ there is a $(t, p)$-conformal measure, then in fact $p = P(t)$.
\end{proof}
\subsection{Characterizations of $\chiinf$ and~$\chisup$}\label{ss:individual pressure}
The following proposition gives some characterizations of~$\chiinf$ and~$\chisup$, which are obtained as direct consequences of the results in~\cite{PrzRivSmi03}.

For each~$\alpha > 0$ put
$$ E_{\alpha} = \bigcap_{n_0 = 1}^{+ \infty} \bigcup_{n = 1}^{+ \infty} B \left( f^n(\Crit(f)), \max \{n_0, n \}^{- \alpha} \right). $$
Observe that the Hausdorff dimension of~$E_\alpha$ is less than or equal to~$\alpha^{-1}$.
It thus follows that the Hausdorff dimension of the set~$E_\infty \= \bigcap_{\alpha > 0} E_\alpha$ is equal to~0.
\begin{prop}\label{p:individual pressure}
For a rational map~$f$ of degree at least two, the following properties hold.
\begin{enumerate}
\item[1.]
Given a repelling periodic point~$p$ of~$f$, let~$m$ be its period and put $\chi(p) \= \tfrac{1}{m} \ln |((f^m)'(p)|$.
Then we have
$$ \inf \{ \chi(p) \mid p \text{ is a repelling periodic point of } f \}
=
\chiinf,$$
$$ \sup \{ \chi(p) \mid p \text{ is a repelling periodic point of } f \}
=
\chisup.$$
\item[2.]
$$ \lim_{n \to + \infty} \tfrac{1}{n} \ln \sup \{ |(f^n)'(z)| \mid z \in \CC \}
=
\chisup. $$
\item[3.]
For each $z_0 \in \CC \setminus E_\infty$ we have
\begin{equation}\label{e:individual minimal pressure}
\lim_{n \to + \infty} \tfrac{1}{n} \ln
\min \{ |(f^n)'(w)| \mid w \in f^{-n}(z_0) \}
=
\chiinf,
\end{equation}
\begin{equation}\label{e:individual maximal pressure}
\lim_{n \to + \infty} \tfrac{1}{n} \ln
\max \{ |(f^n)'(w)| \mid w \in f^{-n}(z_0) \}
=
\chisup.
\end{equation}
\end{enumerate}
\end{prop}
\begin{proof}
\

\partn{1}
The equality involving~$\chiinf$ was shown in \cite[Main Theorem]{PrzRivSmi03}.
To prove the equality involving~$\chisup$, first note that if~$p$ is a repelling periodic point of~$f$, and if we denote by~$m$ its period, then the measure $\mu \= \sum_{j = 0}^{m - 1} \delta_{f^j(p)}$ is invariant by~$f$ and its Lyapunov exponent is equal to~$\chi(p)$.
It thus follows that
$$ \sup \{ \chi(p) \mid p \text{ repelling periodic point of } f \} \le \chisup. $$
The reverse inequality follows from the fact, shown using Pesin theory, that for every ergodic and invariant probability measure~$\mu$ whose Lyapunov exponent is strictly positive and every $\varepsilon > 0$, one can find a repelling periodic point~$p$ such that~$| \chi_\mu - \chi(p)| < \varepsilon$; see for example~\cite[Theorem~10.6.1]{PUbook}.

\partn{2}
For each integer~$n \ge 1$ put
$$ M_n \= \sup \{ |(f^n)'(z)| \mid z \in \CC \}. $$
Note that for integers~$m, n \ge 1$ we have $M_{m + n} \le M_m \cdot M_n$, so the limit
$$ \chi \= \lim_{n \to + \infty} \tfrac{1}{n} \ln M_n $$
exists.
The inequality $\chi \ge \chisup$ follows from part~1.
To prove the reverse inequality, for each integer~$n \ge 1$ let~$z_n \in \CC$ be such that $|(f^n)'(z_n)| = M_n$ and put
$$ \mu_n \= \tfrac{1}{n} \sum_{j = 0}^{n - 1} \delta_{f^j(z_n)}. $$
Let~$(n_j)_{j \ge 0}$ be a diverging sequence of integers so that~$\mu_{n_j}$ converges to a measure~$\mu$, which is invariant by~$\mu$.
Since the function $\ln | f' |$ is bounded from above, the monotone convergence theorem implies that
$$ \lim_{A \to - \infty} \int \max \{ A, \ln |f'| \} d \mu
=
\int \ln |f'|d \mu. $$
On the other hand, for each real number~$A$ we have
$$ \int \max \{ A, \ln |f'| \} d \mu
=
\lim_{j \to + \infty} \int \max \{ A, \ln |f'| \} d \mu_{n_j}
\ge
\limsup_{j \to + \infty} \int \ln |f'| d \mu_{n_j}. $$
We thus conclude that
$$ \chisup \ge \int \ln |f'| d \mu
\ge \limsup_{j \to + \infty} \int \ln |f'| d\mu_{n_j} = \chi. $$

\partn{3}
For a point~$z_0 \in \CC$ which is not in the forward orbit of a critical point of~$f$, the inequalities
$$ \limsup_{n \to + \infty} \tfrac{1}{n} \ln
\min \{ |(f^n)'(w)| \mid w \in f^{-n}(z_0) \}
\le
\chiinf, $$
$$ \liminf_{n \to + \infty} \tfrac{1}{n} \ln
\max \{ |(f^n)'(w)| \mid w \in f^{-n}(z_0) \}
\ge
\chisup. $$
are a direct consequence of part~1 and the following property: For each repelling periodic point~$p$ there is a constant $C > 0$ such that for every integer~$n \ge 1$ there is $w \in f^{-n}(z_0)$ satisfying
$$ C^{-1} \exp(n \chi(p)) \le |(f^n)'(w)| \le C \exp(n \chi(p)). $$

Part~2 shows that for each $z_0 \in \CC$ we have
$$ \limsup_{n \to + \infty} \tfrac{1}{n} \ln
\max \{ |(f^n)'(w)| \mid w \in f^{-n}(z_0) \}
\le
\chisup. $$

It remains to show that for every $z_0 \in \CC \setminus E_\infty$ we have
$$ A(z_0) \= \liminf_{n \to + \infty} \tfrac{1}{n} \ln
\min \{ |(f^n)'(w)| \mid w \in f^{-n}(z_0) \}
\ge
\chiinf. $$
We observe fist that this inequality holds for some point~$z_0$ in~$\CC \setminus E_\infty$.
In fact, let~$K$ be a compact subset of~$J(f)$ of non-zero Hausdorff dimension on which~$f$ is uniformly expanding.
Then for each~$z_0 \in K \setminus E_\infty$ the above inequality follows from part~1 and the ``specification property'' of~\cite[Lemma~3.1]{PrzRivSmi03}.
The final observation is that the function~$A$ is constant on $\C\setminus E_\infty$, see \cite[\S3]{Prz99} or also~\cite[\S1]{PrzRivSmi03}.
To prove this fix~$z, w \in \CC \setminus E_\infty$ and for a given integer join~$z$ to~$w$ with a certain number~$M$ of discs~$(U_j)_{j = 1}^M$ such that $z \in U_1$, $z \in U_M$ and such that for each~$j \in \{1, \ldots, M - 1 \}$ we have $U_j\cap U_{j+1}\neq \emptyset$, in such a way that for each~$j \in \{ 1, \ldots, M \}$ the disc~$2U_j$, with the same center as~$U_j$ and twice the radius, is disjoint from $\bigcup_{i=1,....,n} f^i(\Crit(f))$.
By Koebe Distortion Theorem it follows that there is a constant~$B > 0$ such that the absolute value of the ratio of the derivative of~$f^n$ at corresponding points of~$f^{-n}(z)$ and~$f^{-n}(w)$ is bounded by $\exp(B M)$.
The main point, shown in~\cite[\S3]{Prz99}, is that there is~$\epsilon \in (0, 1)$ such that for every sufficiently large~$n$ such a chain of discs exists for some integer~$M$ satisfying $M\le n^\epsilon$.
In particular, the ratio of these derivatives is sub-exponential with~$n$.
This implies that~$A(z) = A(w)$ and completes the proof of the proposition.
\end{proof}

\section{Nice sets, pleasant couples and induced maps}\label{s:nice induced maps}
In~\S\ref{ss:nice sets and couples} we recall the definition and review some properties of nice sets and couples.
We also introduce a notion weaker than nice couple, that we call ``pleasant couple''.
Then we consider the canonical induced map associated to a pleasant couple in~\S\ref{ss:canonical induced map}, as it was introduced in~\cite[\S4]{PrzRiv07} for nice couples, and review some of its properties (\S\ref{ss:bad pull-backs}).
Finally, we introduce in~\S\ref{ss:two variable pressure} a two variable pressure function associated to the a canonical induced map, that will be important in what follows.

Throughout all this section we fix a rational map~$f$ of degree at least two.
\subsection{Nice sets, nice couples, and pleasant couples}\label{ss:nice sets and couples}
Recall that a neighborhood~$V$ of $\CJ$ is a \emph{nice set for}~$f$, if for every $n \ge 1$ we have $f^n(\partial V) \cap V = \emptyset$, and if each connected component of~$V$ is simply connected and contains precisely one critical point of~$f$ in~$J(f)$.

Let $V = \bigcup_{c \in \CJ} V^c$ be a nice set for~$f$.
Then for every pull-back~$W$ of~$V$ we have either
$$
W \cap V = \emptyset \, \text{ or } \, W \subset V.
$$
Furthermore, if~$W$ and~$W'$ are distinct pull-backs of~$V$, then we have either,
$$
W \cap W' = \emptyset,
\, W \subset W'
\, \text{ or } \,
W' \subset W.
$$
For a pull-back $W$ of $V$ we denote by $c(W)$ the critical point in $\CJ$ and by $m_W \ge 0$ the integer such that $f^{m_W}(W) = V^{c(W)}$.
Moreover we put,
$$
K(V)
=
\{ z \in \CC \mid \text{ for every $n \ge 0$ we have $f^n(z) \not \in V$} \}.
$$
Note that~$K(V)$ is a compact and forward invariant set and for each $c \in \CJ$ the set $V^c$ is a connected component of $\CC \setminus K(V)$.
Moreover, if~$W$ is a connected component of $\CC \setminus K(V)$ different from the~$V^c$, then~$f(W)$ is again a connected component of $\CC \setminus K(V)$.
It follows that~$W$ is a pull-back of~$V$ and that~$f^{m_W}$ is univalent on~$W$.

Given a nice set~$V$ for~$f$ and a neighborhood~$\hV$ of~$\overline{V}$ in~$\CC$ we will say that $(\hV, V)$ is a \emph{pleasant couple for~$f$} if for every pull-back~$W$ of~$V$, the pull-back of~$\hV$ by~$f^{m_W}$ containing~$W$ is
\begin{enumerate}
\item[\puno] contained in~$\hV$ if~$W$ is contained in~$V$;
\item[\pdos] disjoint from~$\Crit(f)$ if~$W$ is disjoint from~$V$.
\end{enumerate}
If~$(\hV, V)$ is a pleasant couple for~$f$, then for each~$c \in \CJ$ we denote by~$\hV^c$ the connected component of~$\hV$ containing~$c$ and for each pull-back~$W$ of~$V$ we will denote by~$\hW$ the pull-back of~$\hV$ by~$f^{m_W}$ that contains~$W$ and put~$m_{\hW} \= m_W$ and~$c(\hW) \= c(W)$.
Thus, when~$W$ is disjoint from~$V$, it is shielded from~$\Crit(f)$ by~$\hW$ by property~\pdos.
Otherwise, $W\subset V$ and then the set~$\hW$ may or may not intersect~$\Crit(f)$.
The latter distinction will be crucial in what follows.

Let~$(\hV, V)$ be a pleasant couple for~$f$ and let~$W$ be a connected component of $\CC \setminus K(V)$.
Then for every $j = 0, \ldots, m_W - 1$, the set~$f^j(W)$ is a connected component of $\CC \setminus K(V)$ different from the~$V^c$.
Thus~$f^j(W)$ is disjoint from~$V$ and by property~\pdos{} the set~$\widehat{f^j(W)}$ is disjoint from~$\Crit(f)$.
It follows that~$f^{m_W}$ is univalent on~$\hW$.

A \emph{nice couple for}~$f$ is a pair $(\hV, V)$ of nice sets for~$f$ such that $\overline{V} \subset \hV$, and such that for every $n \ge 1$ we have~$f^n(\partial V) \cap \hV = \emptyset$.
If $(\hV, V)$ is a nice couple for~$f$, then for every pull-back~$\hW$ of~$\hV$ we have either
$$
\hW \cap V = \emptyset
\, \text{ or } \,
\hW \subset V.
$$
It thus follows that each nice couple is pleasant.

\begin{rema}
The definitions of nice sets and couples given here is slightly weaker than that of~\cite{PrzRiv07,Riv07}.
For a set~$V = \bigcup_{c \in \CJ} V^c$ to be nice, in those papers we required the stronger condition that for each integer $n \ge 1$ we have $f^n(\partial V) \cap \overline{V} = \emptyset$, and that the closures of the sets~$V^c$ are pairwise disjoint.
Similarly, for a pair of nice sets $(\hV, V)$ to be a nice couple we required the stronger condition that for each $n \ge 1$ we have $f^n(\partial V) \cap \overline{\hV} = \emptyset$.
The results we need from~\cite{PrzRiv07} still hold with the weaker property considered here.

Observe that if $(\hV, V)$ is a nice couple as defined here, then~$V$ is a nice set in the sense of~\cite{PrzRiv07,Riv07}.
\end{rema}

The following proposition\footnote{We owe this proposition to Shen, from a personal communication.} sheds some light on the definitions above, although it is not used later on; compare with the construction of nice couples in~\S\ref{ss:non-renormalizable} (Theorem~\ref{t:non-renormalizable}) and in~\cite[\S6]{Riv07}.
\begin{prop}
Suppose that for a rational map~$f$ there exists a nice set $U=\bigcup_{c\in\CJ} U^c$ such that for every integer $n \ge 1$
\begin{equation}\label{e:strongnice}
f^n(\partial U) \cap \overline U = \emptyset.            
\end{equation}
Suppose furthermore that the maximal diameter of a connected component of~$f^{-k}(U)$ converges to~$0$ as $k \to + \infty$.
Then there exists a nice set~$V$ for~$f$ that is compactly contained in~$U$ such that $(U,V)$ is a nice couple for~$f$.
\end{prop}
\begin{proof}
Since~$U$ is a nice set each connected component of the set $A \= \CC \setminus f^{-1}(K(U))$ is a pull-back of~$U$.
Furthermore, by~\eqref{e:strongnice} each connected component~$W$ of~$A$ intersecting~$U$ is compactly contained in~$U$, and~$m_W$ is the first return time to~$U$ of points in~$W$.

If the forward trajectory of~$c$ visits~$U$, take as $V^c$ the connected component containing~$c$ of~$A$.
Since~$U$ is a nice set, $V^c$ is a first return pull-back of~$U$, and by~\eqref{e:strongnice} the set~$V^c$ is compactly contained in~$U$.
In particular for each integer $n \ge 1$ we have $f^n(\partial V^c) \cap U = \emptyset$.
For each critical point $c\in\CJ$ whose forward trajectory never returns to~$U$, take a preliminary disc~$D$ compactly contained in~$U^c$.
By~\eqref{e:strongnice} each connected component of~$A$ intersecting~$D$ is compactly contained in~$U^c$. 
Let now~$\tV^c$ be the union of~$D$ and all those connected components of~$A$ intersecting~$D$.
The hypothesis on diameters of pull-backs implies that~$V^c$ is compactly contained in~$U$, and that each point in~$\partial \tV^c$ is either contained in~$\partial D \cap (\CC \setminus A)$, or in the boundary of a connected component of~$A$ intersecting~$D$ (which is a first return pull-back of~$U$).
Therefore for each integer~$n \ge 1$ we have $f^n(\partial \tV^c) \cap U = \emptyset$.
Finally let~$V^c$ be the union of~$\tV^c$ and all connected components of $\CC \setminus \tV^c$ contained in~$U^c$ (We do this "filling holes" trick since \emph{a priori} it could happen that the union of~$D$ and one of the connected components of~$A$, and consequently~$\tV^c$, might not be simply-connected).
We have~$\partial V^c \subset \partial \tV^c$, so for each integer~$n \ge 1$ we have $f^n(\partial V^c) \cap U = \emptyset$.

Set $V=\bigcup_{c\in\CJ} V^c$.
We have shown that for each integer~$n \ge 1$ we have $f^n(\partial V) \cap U = \emptyset$, so $(U,V)$ is a nice couple.
\end{proof}
\subsection{Canonical induced map}\label{ss:canonical induced map}
Let $(\hV, V)$ be a pleasant couple for~$f$.
We say that an integer $m \ge 1$ is a \emph{good time for a point $z$} in $\CC$, if $f^m(z) \in V$ and if the pull-back of $\hV$ by $f^m$ to $z$ is univalent.
Let $D$ be the set of all those points in $V$ having a good time
and for $z \in D$ denote by $m(z) \ge 1$ the least good time of
$z$. Then the map $F : D \to V$ defined by $F(z) \= f^{m(z)}(z)$
is called {\it the canonical induced map associated to $(\hV,
V)$}.
We denote by $J(F)$ the maximal invariant set of~$F$ and by~$\fD$ the collection of connected components of~$D$.

As~$V$ is a nice set, it follows that each connected component~$W$ of~$D$ is a pull-back of~$V$ and that for each~$z \in W$ we have~$m(z) = m_W$.
Moreover,~$f^{m_W}$ is univalent on~$\hW$ and by property~\puno{} of pleasant couples we have $\hW \subset \hV$.
Similarly, for each integer~$n \ge 1$, each connected component~$W$ of the
domain of definition of~$F^n$ is a pull-back of~$V$ and~$f^{m_W}$
is univalent on~$\hW$.
Conversely, if~$W$ is a pull-back of~$V$ strictly contained in~$V$ such that~$f^{m_W}$ is univalent on~$\hW$,
then there is $c \in \CJ$ and an integer~$n \ge 1$ such that~$F^n$ is defined on~$W$ and $F^n(W) = V^c$.
Indeed, in this case~$m_W$ is a good time for each element of~$W$ and therefore $W \subset D$.
Thus, either we have $F(W) = V^{c(W)}$ and then~$W$ is a connected component of $D$, or~$F(W)$ is a pull-back of~$V$ strictly contained in~$V$ such that~$f^{m_{F(W)}}$ is univalent on~$\widehat{F(W)}$.
Thus, repeating this argument we can show by induction that there is an integer~$n \ge 1$ such that~$F^n$ is defined on~$W$ and $F^n(W) = V^{c(W)}$.

The following result was shown in~\cite{PrzRiv07} for nice couples.
The proof applies without change to pleasant couples.
\begin{lemm}[\cite{PrzRiv07}, Lemma~4.1]\label{l:mixingness}
For every rational map~$f$ there is $r > 0$ such that if $(\hV, V)$ is a pleasant couple satisfying
\begin{equation}\label{e:diam bound}
\max_{c \in \CJ} \diam\left( \hV^c \right) \le r,
\end{equation}
then the canonical induced map $F : D \to V$ associated to $(\hV, V)$ is topologically mixing on $J(F)$.
Moreover there is $\widetilde{c} \in \CJ$ such that the set
\begin{equation}\label{e:return times}
\left\{ m_W \mid W \in \fD \text{ contained in $V^{\widetilde{c}}$ and such that $F(W) = V^{\widetilde{c}}$} \right\}
\end{equation}
is non\nobreakdash-empty and its greatest common divisor is equal to~$1$.
\end{lemm}
\begin{rema}\label{r:using symbolic}
We will apply several results of~\cite{MauUrb03} to the induced map~$F$.
However, most of the results we need from~\cite{MauUrb03} are stated for the associated symbolic space.
The corresponding results for the induced map~$F$ can be obtained using Lemma~3.1.3, Proposition~3.1.4 and Theorem~4.4.1 of~\cite{MauUrb03}.
\end{rema}

\subsection{Bad pull-backs}\label{ss:bad pull-backs}
We will now introduce the concept of ``bad pull-backs'' of a pleasant couple.
It is an adaptation of the concept with the same name introduced in~\cite[\S7.1]{PrzRiv07} for a nice set.

Given a pleasant couple~$(\hV, V)$ and an integer~$n \ge 1$, a point~$y \in f^{-n}(V)$ is a \emph{bad iterated pre-image of order~$n$} if for every~$j \in \{ 1, \ldots, n \}$ such that~$f^j(y) \in V$ the map~$f^j$ is not univalent on the pull-back of~$\hV$ by~$f^j$ containing~$y$.
In this case every point~$y'$ in the pull-back~$X$ of~$V$ by~$f^n$ containing~$y$ is a bad iterated pre-image of order~$n$.
We could call~$X$ a bad pull-back of~$V$ of order~$n$, although this terminology would not be used in the rest of the paper, see Figure~\ref{f:bad} below.
Furthermore, a connected component~$\badp$ of~$f^{-n}(\hV)$ is a \emph{bad pull-back of~$\hV$ of order~$n$}, if it contains a bad iterated pre-image of order~$n$.\footnote{When the pleasant couple~$(\hV, V)$ is nice, it is easy to see that the notion of bad pull-back as defined here coincides with that of~\cite{PrzRiv07}.
See also Remark~\ref{r:bad pull-backs}.}

The following two lemmas will be used in the proof of the Main Theorem in~\S\ref{s:proof of nice thermodynamics}.
They are adaptations to pleasant couples of part~1 of Lemma~7.4 and of Lemma~7.1 of~\cite{PrzRiv07}.
Given a pleasant couple~$(\hV, V)$ for~$f$ we denote by~$\fL_V$ the collection of all the connected components of~$\CC \setminus K(V)$.
On the other hand, let~$\badp$ be a pull-back of~$\hV$ and recall that~$m_{\badp} \ge 0$ denotes the integer such that~$f^{m_{\badp}}(\badp)$ is equal to a connected component of~$\hV$.
Then we let~$\fD_{\badp}$ be the collection of all the pull-backs~$W$ of~$V$ that are contained in~$\badp$, such that~$f^{m_W}$ maps the pull-back~$\hW$ of~$\hV$ by~$f^{m_W}$ containing~$W$ univalently onto~$\hV^{c(W)}$, such that~$f^{m_{\badp}}(W) \subset V$ and such that $f^{m_{\badp} + 1}(W) \in \fL_V$.
See~Figure~\ref{f:bad}.
\begin{figure}[hbt]
\begin{center}
\input{bad.pstex_t}
\end{center}
\caption{On the left, the family associated to a connected component~$\badp$ of~$\hV$; on the right, the one associated to a higher order pull-back~$\badp$ of~$\hV$.}\label{f:bad}
\end{figure}
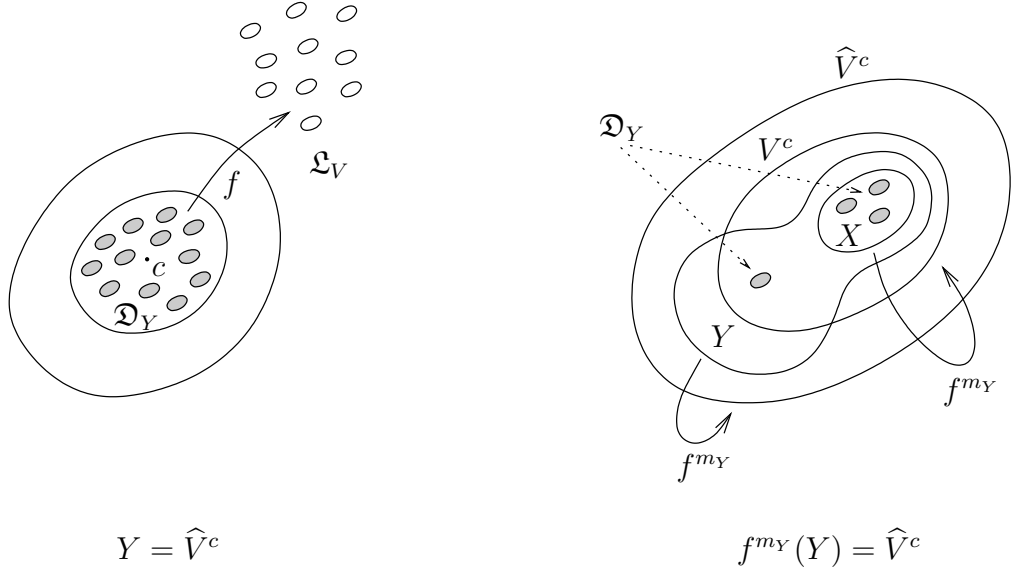
\begin{lemm}[\cite{PrzRiv07}, part~1 of Lemma~$7.4$]\label{l:decomposition into bad}
Let~$(\hV, V)$ be a pleasant couple for~$f$ and let~$\fD$ be the collection of the connected components of~$D$.
Then
$$ \fD
\subset
\left( \bigcup_{c \in \CJ} \fD_{\hV^c} \right)
\cup
\left( \bigcup_{\badp \text{ bad pull-back of~$\hV$}} \fD_{\badp} \right). $$
\end{lemm}
\begin{proof}
Let~$W \in \fD$.
If~$f(W) \in \fL_V$, then there is~$c \in \CJ$ such that~$W \in \fD_{\hV^c}$.
Suppose now~$f(W) \not \in \fL_V$, so there is an integer~$j \in \{1, \ldots, m_W - 1 \}$ such that $f^j(W) \subset V$.
Let~$n$ be the largest such integer, so we have~$f^{n + 1}(W) \in \fL_V$ and, if we let~$\badp$ be the pull-back of~$\hV$ by~$f^n$ containing~$W$, then~$W \in \fD_{\badp}$.
It remains to show that~$\badp$ is a bad pull-back of~$\hV$ of order~$n$.
Just observe that if we fix~$y \in W \subset \badp$, then~$m_W$ is the least good time of~$y$, so~$y$ is a bad iterated pre-image of~$f^n(y) \in f^n(W) \subset V$ of order~$n$ and~$\badp$ is a bad pull-back of~$\hV$ of order~$n$.
This completes the proof of the lemma.
\end{proof}
\begin{lemm}[\cite{PrzRiv07}, Lemma~7.1]\label{l:bad pull-backs}
Let~$f$ be a rational map, let~$(\hV, V)$ be a pleasant couple for~$f$ and let $L \ge 1$ be such that for every~$\ell \in \{ 1, \ldots, L \}$ the set~$f^\ell(\CJ)$ is disjoint from~$\hV$.
Then for each~$z_0 \in V$ and each integer $n \ge 1$, the number of bad iterated pre-images of~$z_0$ of order~$n$ is at most
$$(2 L \deg(f) \#(\CJ))^{n/L}. $$
In particular, the number of bad pull-backs of~$\hV$ of order~$n$ is at most
$$ \#(\CJ) (2 L \deg(f) \#(\CJ))^{n/L}. $$
\end{lemm}
\begin{proof}
\

\partn{1}
Given an integer~$n \ge 1$ and a bad iterated pre-image~$y$ of order~$n$, let~$\ell(y, n) \in \{ 0, \ldots, n - 1 \}$ be the largest integer such that the pull-back of~$\hV$ by~$f^{n - \ell(y, n)}$ containing~$f^{\ell(y, n)}(y)$ intersects~$\CJ$.
Using property~\pdos{} of pleasant couples we obtain that~$f^{\ell(y, n)}(y) \in V$, and in the case where~$\ell(y, n) > 0$, that the point~$y$ is a bad iterated pre-image of order~$\ell(y, n)$.

\partn{2}
Given an integer~$n \ge 1$ and a bad iterated preimage~$y$ of order~$n$, define~$k \ge 1$ and a strictly decreasing sequence of non-negative integers~$(\ell_0, \ldots, \ell_k)$ by induction as follows.
We put~$\ell_0 = n$ and suppose that for some integer~$j \ge 0$ the integer~$\ell_j$ is already defined in such a way that~$f^{\ell_j}(y) \in V$.
If~$\ell_j = 0$ then define~$k \= j$ and stop.
Otherwise we have~$f^{\ell_j}(y) \in V$ by the induction hypothesis and therefore~$y$ is a bad iterated pre-image of~$f^{\ell_j}(y)$ of order~$\ell_j$.
Then we define~$\ell_{j + 1} \= \ell(y, \ell_j)$.
As remarked above~$f^{\ell_{j + 1}}(y) = f^{\ell(y, \ell_j)}(y) \in V$, so the induction hypothesis is satisfied.

\partn{3}
Fix an integer~$n \ge 1$.
To each bad iterated preimage~$y$ of order~$n$ we have associated in part~2 an integer~$k \ge 1$ and a strictly decreasing sequence of non-negative integers~$(\ell_0, \ldots, \ell_k)$.
We have~$\ell_0 = n$, $\ell_k = 0$ and for each~$j \in \{ 1, \ldots, k \}$ the pull-back of~$\hV$ by~$f^{\ell_{j - 1} - \ell_j}$ containing~$f^{\ell_j}(y)$ contains an element~$c$ of~$\CJ$ and at most~$\deg_f(c)$ elements of~$f^{-(\ell_{j - 1} - \ell_j)}(f^{\ell_j}(y))$. 
As for each~$c \in \CJ$ and each integer~$m \ge 1$ there are at most~$\#(\CJ)$ connected components of~$f^{-m}(\hV^c)$ intersecting~$\CJ$, it follows that there are at most
$$ (\deg(f)\#(\CJ))^{k} $$
bad iterated pre-images of~$z_0$ of order~$n$ whose associated sequence is equal to~$(\ell_0, \ldots, \ell_k)$.

By definition of~$L$ for each~$j \in \{ 1, \ldots, k \}$ we have~$\ell_{j - 1} - \ell_j \ge L$.
So~$k \le n / L$ and for each integer~$m \in \{ 1, \ldots, n \}$ there is at most one integer~$r \in \{ 0, \ldots, L - 1 \}$ such that~$m + r$ is one of the~$\ell_j$.
Thus there are at most~$(L + 1)^{n / L}$ such decreasing sequences.

We conclude that the number of bad iterated pre-images of~$z_0$ of order~$n$ is at most,
$$ (L + 1)^{n / L}(\deg(f)\#(\CJ))^{n/L}
\le (2L \deg(f) \#(\CJ))^{n/L}. $$
\end{proof}
\begin{rema}
\label{r:bad pull-backs}
The purpose of this remark is to show the reverse inclusion of Lemma~\ref{l:decomposition into bad}.
Although this is not used in the proof of the Main Theorem, we think the argument is useful to understand bad pull-backs of pleasant couples.
As for each~$c \in \CJ$ we clearly have~$\fD_{\hV^c} \subset \fD$, we just need to show that for each integer~$n \ge 1$ and each bad pull-back~$\badp$ of~$\hV$ by~$f^n$ we have~$\fD_{\badp} \subset \fD$.
To do this, let~$y \in \badp$ be a bad iterated pre-image of order~$n$ and let~$k \ge 1$ and~$(\ell_0, \ldots, \ell_k)$ be as in the proof of Lemma~\ref{l:bad pull-backs}.
An inductive argument using property~\pdos{} of pleasant couples shows that for each~$j \in \{0, \ldots, k - 1 \}$ the set~$f^{\ell_j}(\badp) \cap f^{-(n - \ell_j)}(V)$ is contained in~$V$.
In particular, each element~$W$ of~$\fD_{\badp}$ is contained in~$V$.
On the other hand, observe that the pull-back of~$\hV$ by~$f^{m_W - n}$ containing~$f^n(W)$ is univalent and by property~\puno{} of pleasant couples it is contained in~$\hV$.
Thus, to show~$W \in \fD$, we just need to show that each element~$y'$ of~$W$ is a bad iterated pre-image of order~$n$.
Let~$i \in \{ 1, \ldots, n \}$ be such that~$f^i(y') \in V$ and let~$j \in \{ 0, \ldots, k - 1 \}$ be the largest integer such that~$\ell_j \ge i$.
By property~\puno{} of pleasant couples it follows that the pull-back of~$\hV$ by~$f^{\ell_{j} - i}$ containing~$f^i(y')$ is contained in~$\hV$.
Since the pull-back of~$\hV$ by~$f^{\ell_{j} - \ell_{j + 1}}$ containing~$f^{\ell_{j + 1}}(y')$ intersects~$\Crit(f)$, it follows the pull-back of~$\hV$ by~$f^{i - \ell_{j + 1}}$ containing~$f^{\ell_{j + 1}}(y')$ intersects~$\Crit(f)$ and hence that that~$f^i$ is not univalent on the pull-back of~$\hV$ by~$f^i$ containing~$y$.
This completes the proof that~$y'$ is a bad iterated pre-image of order~$n$ and that~$W \in \fD$.
\end{rema}
\subsection{Pressure function of the canonical induced map}\label{ss:two variable pressure}
Let $(\hV, V)$ be a pleasant couple for~$f$ and let $F : D \to V$ be the canonical induced map associated to $(\hV, V)$.
Furthermore, denote by $\fD$ the collection of connected components of~$D$ and
for each $c \in \CJ$ denote by $\fD^c$ the collection of all
elements of $\fD$ contained in $V^c$, so that $\fD = \bigsqcup_{c
\in \CJ} \fD^c$. A word on the alphabet $\fD$ will be called
\emph{admissible} if for every pair of consecutive letters $W,
W' \in \fD$ we have $W \in \fD^{c(W')}$. For a given integer $n
\ge 1$ we denote by~$E^n$ the collection of all admissible words
of length~$n$. Given $W \in \fD$, denote by $\phi_W$ the
holomorphic extension to $\hV^{c(W)}$ of the inverse of $F|_W$.
For a finite word $\underline{W} = W_1 \ldots W_n \in E^*$ put
$c(\underline{W}) \= c(W_n)$ and $m_{\underline{W}} = m_{W_1} +
\cdots + m_{W_n}$. Note that the composition
$$
\phi_{\underline{W}} \= \phi_{W_1} \circ \cdots \circ \phi_{W_n}
$$
is well defined and univalent on $\hV^{c(\uW)}$ and takes images
in $V$.

For each $t, p \in \R$ and $n \ge 1$ put
$$
Z_n(t, p) \= \sum_{\underline{W} \in E^n} \exp(- m_{\underline{W}}
p)\left( \sup \left\{ |\phi_{\underline{W}}'(z)| \mid z \in
V^{c(\uW)} \right\} \right)^t.
$$
It is easy to see that for a fixed $t, p \in \R$ the sequence $( \ln Z_n(t, p) )_{n \ge 1}$ is sub-additive, and hence
that we have
\begin{equation}\label{e:induced pressure}
P(F, - t \ln |F'| - p m) := \lim_{n \to + \infty} \tfrac{1}{n} \ln
Z_n(t, p) = \inf \left\{ \tfrac{1}{n} \ln Z_n(t, p) \mid  n \ge 1
\right\},
\end{equation}
see for example Lemma~2.1.1 and Lemma~2.1.2 of~\cite{MauUrb03}.
Here~$m$ is the function defined in~\S\ref{ss:canonical induced map}, that to each point $z \in D$ it associates the
least good time of~$z$.
The number~\eqref{e:induced pressure} is called the \emph{pressure function of~$F$ for the potential $- \ln |F'| - pm$}.
It is easy to see that for every $t, p \in \R$ the sequence $( \tfrac{1}{n} \ln Z_n(t, p))_{n \ge 1}$ is
uniformly bounded from below, so that~\eqref{e:induced pressure}
does not take the value~$-\infty$. Note however that if~$D$ has
infinitely many connected components and we take~$t = 0$ and $p = 0$, then we have $P(F, 0) = + \infty$.

When applying the results of~\cite{MauUrb03} to the induced map~$F$ we will use the fact that the function~$-\ln |F'|$ defines a H{\"o}lder function on the associated symbolic space and hence that for each~$(t, p) \in \R^2$ the same holds for the function~$- t \ln |F'| - pm$.

The following property will be important to use the results of~\cite{MauUrb03}.
\begin{enumerate}
\item[(*)]
There is a constant~$C_M > 0$ such that for every~$\kappa \in (0, 1)$ and every ball~$B$ of~$\CC$, the following property holds.
Every collection of pairwise disjoint sets of the form~$D_{\underline{W}}$, with $\underline{W} \in E^*$, intersecting~$B$ and with diameter at least $\kappa \cdot \diam(B)$, has cardinality at most~$C_M \kappa^{-2}$.
\end{enumerate}
In fact,~$F$ determines a Conformal Graph Directed Markov System (CGDMS) in the sense of~\cite{MauUrb03}, except maybe for the ``cone property''~(4d).
But in~\cite{MauUrb03} the cone property is only used in~\cite[Lemma~4.2.6]{MauUrb03} to prove~(*).
Thus, when property~(*) is satisfied all the results of~\cite{MauUrb03} apply to~$F$.
In \cite[Proposition~A.2]{PrzRiv07} we have shown that property~(*) holds when the pleasant couple $(\hV, V)$ is nice.

The function,
\begin{center}
\begin{tabular}{crcl}
$\pressure$ : & $\R^2$ & $\to$ & $\R \cup \{ + \infty \}$  \\
& $(t, p)$ & $\mapsto$ & $P(F, - t \ln |F'| - p m)$,
\end{tabular}
\end{center}
will be important in what follows.
Notice that if~$\pressure$ is finite at $(t_0, p_0) \in \R^2$, then by Proposition~2.1.9 the function $- t_0 \ln |F'| - p_0 m$ defines a summable H{\"o}lder potential on the symbolic space associated to the induced map~$F$.
Furthermore, it follows that~$\pressure$ is finite on the set
$$ \left\{ (t, p) \in \R^2 \mid t \ge t_0, p \ge p_0 \right\} $$
and, restricted to the set where it is finite, the function~$\pressure$ it is strictly decreasing on each of its variables.
\begin{lemm}\label{l:pressure}
Let~$f$ be a rational map of degree at least two and let $(\hV, V)$ be a
pleasant couple for~$f$ satisfying property~(*).
Then the function~$\pressure$ defined above satisfies the following properties.
\begin{enumerate}
\item[1.]
The function~$\pressure$ is real analytic on the interior of the set where it is finite.
\item[2.]
The function~$\pressure$ is strictly negative on $\{ (t, p) \in \R^2 \mid p > P(t) \}$.
\end{enumerate}
\end{lemm}
\begin{proof}
\

\partn{1}
If~$\pressure(t, p) < + \infty$, then by~\cite[Proposition~2.1.9]{MauUrb03} the function~$- t \ln |F'| - pm$ defines a summable H\"older potential on symbolic space associated to~$F$.
Thus the desired result follows from~\cite[Theorem~2.6.12]{MauUrb03}, see~Remark~\ref{r:using symbolic}.

\partn{2}
Let $(t_0, p_0) \in \R^2$ be such that $p_0 > P(t_0)$.
Then for each point $z_0 \in V$ for which~\eqref{e:tree pressure} holds, we have
\begin{multline*}\label{e:full tree of f}
\sum_{k = 1}^{+ \infty} \sum_{y \in F^{-k}(z)} \exp( - p_0 m(y) ) |(F^k)'(y)|^{-t_0}
\\ \le
\sum_{n = 1}^{+ \infty} \exp( - p_0 n) \sum_{y \in f^{-n}(z_0)} |(f^n)'(y)|^{-t_0} < + \infty,
\end{multline*}
which implies that $\pressure(t_0, p_0) \le 0$.
This shows that the function~$\pressure$ is non\nobreakdash-positive on $\{ (t, p) \in (0, + \infty) \times \R \mid p > P(t) \}$.
That~$\pressure$ is strictly negative on this set follows from the fact that, on this set, $\pressure$ is strictly decreasing on each of its variables.
\end{proof}

\section{From the induced map to the original map}\label{s:lifting}
The purpose of this section is to prove the following theorem, which gives us some sufficient conditions to obtain the conclusions of the Main Theorem.

We denote by~$\Jcon(f)$ the ``conical Julia set'' of~$f$, which is defined in~\S\ref{ss:conical conformal}.
Recall that conformal measures were defined in~\S\ref{ss:conformal measures}.
\begin{theoalph}\label{t:lifting}
Let~$f$ be a rational map of degree at least two, let $(\hV, V)$ be a pleasant couple for~$f$ satisfying property~(*), and let~$\pressure$ be the corresponding pressure function defined in~\S\ref{ss:two variable pressure}.
Then for each~$t_0 \in (\tneg, + \infty)$, the following properties hold.
\begin{description}
\item[Conformal measure]
If~$\pressure$ vanishes at $(t, p) = (t_0, P(t_0))$, then there is a unique $(t_0, P(t_0))$\nobreakdash-conformal probability measure for~$f$.
Moreover this measure is non\nobreakdash-atomic, ergodic, and it is supported on~$\Jcon(f)$.
\item[Equilibrium state]
If~$\pressure$ is finite on a neighborhood of $(t, p) = (t_0, P(t_0))$, and vanishes at this point, then there is a unique equilibrium measure of~$f$ for the potential~$-t_0 \ln|f'|$.
Furthermore, this measure is ergodic, absolutely continuous with
respect to the unique $(t_0, P(t_0))$\nobreakdash-conformal
probability measure of~$f$, and its density is bounded from below
by a strictly positive constant almost everywhere.
If furthermore~$(\hV, V)$ satisfies the conclusions of Lemma~\ref{l:mixingness}, then the equilibrium state is exponentially mixing and it satisfies the Central Limit theorem.
\item[Analyticity of the pressure function]
If~$\pressure$ is finite on a neighborhood of~$(t, p) = (t_0, P(t_0))$ and for each~$t \in \R$ close to~$t_0$ we have $\pressure(t, P(t)) = 0$, then the pressure function~$P$ is real analytic on a neighborhood of~$t = t_0$.
\end{description}
\end{theoalph}
In~\S\S\ref{s:Whitney decomposition}, \ref{s:pull-back contribution}, \ref{s:proof of nice thermodynamics} we verify that, for a map as in the Main Theorem or more generally as in Theorem~\ref{t:nice thermodynamics} in~\S\ref{s:proof of nice thermodynamics}, and for a given~$t_0 \in (\tneg, \tpos)$ the function~$\pressure$ corresponding to a sufficiently small pleasant couple is finite on a neighborhood of~$(t, p) = (t_0, P(t_0))$ and that for each~$t \in \R$ close to~$t_0$ we have~$\pressure(t, P(t)) = 0$.

After some general considerations in~\S\ref{ss:conical conformal}, the assertions about the conformal measure are shown in~\S\ref{ss:conformal for rational}.
The assertions concerning the equilibrium state are shown in~\S\ref{ss:equilibrium}, and the analyticity of the pressure function is shown in~\S\ref{ss:lifting analyticity}.

Throughout the rest of this section we fix~$f$, $(\hV, V)$, $F$, $\pressure$ as in the statement of the theorem.
\subsection{The conical Julia set and sub-conformal measures}\label{ss:conical conformal}
The \emph{conical Julia set} of~$f$, denoted by $\Jcon(f)$, is by definition the set of all those points~$x$ in $J(f)$ for which there exists~$\rho(x) > 0$ and an arbitrarily large positive integer~$n$, such that the pull-back of the ball $B(f^n(x), \rho(x))$ to~$x$ by $f^n$ is univalent.
This set is also called \emph{radial Julia set}.

We will use the following general result, which is a strengthened version of ~\cite[Theorem~$5.1$]{McM00}, \cite[Theorem~$1.2$]{DenMauNitUrb98}, with the same proof.
Given $t, p \in \R$ we will say that a Borel measure $\mu$ is $(t, p)$\nobreakdash-\emph{sub-conformal} $f$, if for every Borel subset~$U$ of~$\CC \setminus \Crit(f)$ on which~$f$ is injective we have
\begin{equation}\label{e:sub-conformal}
\exp(p) \int_U |f'|^t d\mu \le \mu(f(U)).
\end{equation}
\begin{prop}\label{p:conical conformal}
Fix $t \in (\tneg, +\infty)$ and $p \in [P(t), + \infty)$.
If~$\mu$ is a $(t, p)$\nobreakdash-sub-conformal measure for~$f$ supported on $\Jcon(f)$, then $p = P(t)$, the measure~$\mu$ is $(t, P(t))$\nobreakdash-conformal, and every other $(t, P(t))$\nobreakdash-conformal measure is proportional to~$\mu$.
Moreover, every subset~$X$ of~$\CC$ such that $f(X) \subset X$ and $\mu(X) > 0$ has full measure with respect to~$\mu$.
\end{prop}
The proof of this proposition depends on the following lemma.
\begin{lemm}\label{l:conical conformal}
Let $t, p \in \R$ and let $\mu$ be a $(t, p)$\nobreakdash-sub\nobreakdash-conformal measure supported on $\Jcon(f)$.
Suppose that for some $p' \le p$ there exists a non\nobreakdash-zero $(t, p')$\nobreakdash-conformal measure~$\nu$ that is
supported on $J(f)$. Then $p' = p$ and $\mu$ is absolutely
continuous with respect to $\nu$. In particular $\nu(\Jcon(f)) >
0$.
\end{lemm}
\begin{proof}
For $\rho > 0$ put $\Jcon(f, \rho) \= \{ x \in \Jcon(f) \mid \rho(x) \ge \rho \}$, so that $\Jcon(f) = \bigcup_{\rho > 0} \Jcon(f, \rho)$.
For each $\rho_0 > 0$, Koebe Distortion Theorem implies that there is a constant $C > 1$ such that for every $x \in \Jcon(f, \rho_0)$ there are arbitrarily small $r > 0$, so that for some integer $n \ge 1$ we
have,
\begin{equation}\label{e:conical}
\mu(B(x, 5r)) \le C \exp(-np) r^t \ \mbox{ and } \ \nu(B(x, r))
\ge C^{-1} \exp(-np') r^{t}.
\end{equation}
Given a subset $X$ of $\Jcon(f, \rho_0)$, by Vitali's covering lemma, for every $r_0 > 0$ we can find a collection of pairwise disjoint
balls $(B(x_j, r_j))_{j > 0}$ and strictly positive integers $(n_j)_{j > 0}$, such that $x_j \in X$, $r_j \in (0, r_0)$, $X \subset \bigcup_{j > 0} B(x_j, 5 r_j)$ and such that for each $j > 0$ the inequalities~\eqref{e:conical} hold for $x \= x_j$ and $r \= r_j$ and $n = n_j$.
Moreover, for each integer~$n_0 \ge 1$ we may choose~$r_0$ sufficiently small so that for each~$j > 0$ we have $n_j \ge n_0$.
Since by hypothesis $p' \le p$, we obtain
$$
\nu(X) \ge C^{-2} \exp(n_0(p - p')) \mu(X).
$$

Suppose by contradiction that $p' < p$. Choose $\rho_0 > 0$ such
that $\mu(\Jcon(f, \rho_0)) > 0$ and set $X \= \Jcon(f, \rho_0)$.
As in the inequality above $n_0 > 0$ can by taken arbitrarily
large, we obtain a contradiction. So $p' = p$ and it follows that
$\mu$ is absolutely continuous with respect to $\nu$.
\end{proof}
\begin{proof}[Proof of Proposition~\ref{p:conical conformal}]
Let~$\nu$ be a $(t, P(t))$\nobreakdash-conformal measure~$\nu$ for~$f$ supported on $J(f)$.
By~\cite[Theorem~A and Theorem~A.7]{PrzRivSmi04} there is at least one such measure, see also~\cite{Prz99}.
So Lemma~\ref{l:conical conformal} implies that $p = P(t)$, and that~$\mu$ is absolutely continuous with respect to~$\nu$.

In parts~$1$ and~$2$ we show that~$\nu$ is proportional to~$\mu$.
It follows in particular that~$\mu$ is conformal.
In part~3 we complete the proof of the proposition by showing the last statement of the proposition.

\partn{1}
First note that $\nu' \= \nu|_{\CC \setminus \Jcon(f)}$ is a
conformal measure for $f$ of the same exponent as $\nu$. Then
Lemma~\ref{l:conical conformal} applied to $\nu = \nu'$ implies
that, if $\nu'$ is non\nobreakdash-zero, then $\nu'(\Jcon(f)) > 0$. This
contradiction shows that $\nu'$ is the zero measure and that~$\nu$ is supported on~$\Jcon(f)$.

\partn{2}
Denote by~$g$ the density~$\mu$ with respect to~$\nu$.
Since~$\nu$ is conformal and~$\mu$ sub-conformal, the function~$g$ satisfies $g \circ f \ge g$ on a set of full $\nu$-measure.
Let $\delta > 0$ be such that $\nu(\{ g \ge \delta
\}) > 0$.
As $\nu$ is supported on $\Jcon(f)$, there is a density point of $\{g \ge \delta \}$ for $\nu$ that belongs to $\Jcon(f)$.
Going to large scale and using $g \circ f \ge g$, we conclude that
$\{ g \ge \delta \}$ contains a ball of definite size, up to a set
of $\nu$-measure~$0$.
It follows by the locally eventually onto property of~$f$ on~$J(f)$ that the set $\{ g \ge \delta \}$ has full measure with respect to~$\nu$.
This implies that $g$ is constant $\nu$-almost everywhere and therefore that~$\nu$ and~$\mu$ are proportional.
In particular~$\mu$ is conformal.

\partn{3}
Suppose that~$X$ is a Borel subset of~$\CC$ of positive measure with respect to~$\mu$ and such that $f(X) \subset X$.
Then the restriction~$\mu|_X$ of~$\mu$ to~$X$ is a $(t, P(t))$\nobreakdash-sub-conformal measure supported on the conical Julia set.
It follows that $\mu|_X$ is proportional to~$\mu$, and thus that $\mu|_X = \mu$ and that~$X$ has full measure with respect to~$\mu$.
\end{proof}
\subsection{Conformal measure}\label{ss:conformal for rational}
Given $t, p \in \R$ we will say that a measure~$\mu$ supported on the maximal invariant set~$J(F)$ of~$F$ is $(t, p)$\nobreakdash-\emph{conformal for}~$F$ if for every Borel subset~$U$ of a connected component~$W$ of~$D$ we have
$$
\mu(F(U)) = \exp(pm_W) \int_U |F'|^t \, d\mu.
$$
In view of~\cite[Proposition~2.1.9]{MauUrb03}, the hypothesis that
$$
P(F, -t_0 \ln|F'| - P(t_0)m)
=
\pressure(t_0, P(t_0))
=
0
$$
implies that~$-t_0 \ln|F'| - P(t_0)m$ defines a summable H{\"o}lder potential on the symbolic space associated to~$F$.
Furthermore, by Theorem~3.2.3 and Proposition~4.2.5 of~\cite{MauUrb03} it follows that the induced map~$F$ admits a non\nobreakdash-atomic $(t_0, P(t_0))$\nobreakdash-conformal measure supported on~$J(F)$.
Therefore the assertions in~Theorem~\ref{t:lifting} about conformal measures are a direct consequence of Proposition~\ref{p:conical conformal} and of the following proposition.
\begin{prop}\label{p:lifting conformal}
Let~$F$ be the canonical induced map associated to a pleasant couple $(\hV, V)$ for~$f$ that satisfies property~(*).
Then for every $t \in (\tneg, + \infty)$ and $p \in [P(t), + \infty)$, each $(t, p)$\nobreakdash-conformal measure of~$F$ is in fact $(t, P(t))$-conformal, and it is the restriction to~$V$ of a non\nobreakdash-atomic $(t, p)$\nobreakdash-conformal measure of~$f$ supported on $\Jcon(f)$.
\end{prop}
\begin{proof}
The proof of this proposition is a straight forward generalization of that of~\cite[Proposition~B.2]{PrzRiv07}.
We will only give a sketch of the proof here.

Since~$t > \tneg$ there is a $(t, P(t))$\nobreakdash-conformal measure~$\hmu$ for~$f$ whose topological support is equal to the whole Julia set of~$f$ (Proposition~\ref{p:conformal measures}).
Let~$\fL_V$ be the collection of connected components of $\CC \setminus K(V)$.
Notice that for each $W \in \fL_V$ we have $\hmu(W) \sim \exp(-m_W P(t)) \diam(W)^t$, for an implicit constant independent of~$W$.

Let~$\mu$ be a $(t, p)$-conformal measure for~$F$.
For each $W \in \fL_V$ denote by $\phi_W : \hV^{c(W)} \to \hW$ the inverse of $f^{m_W}|_{\hW}$, and let $\mu_W$ be the measure supported on $W$, defined by
$$
\mu_W(X) = \exp(-m_W p) \int_{f^{m_W}(X \cap W)} |\phi_W'|^t d\mu.
$$
Clearly the measure $\sum_{W \in \fL_V} \mu_W$ is supported on~$\Jcon(f)$, non\nobreakdash-atomic, and for each $W \in \fL_V$ we have $\mu_W(\CC) \sim \exp(-m_Wp) \diam(W)^t$.
Since we also have $\hmu(W) \sim \exp(-m_W P(t)) \diam(W)^t$, and $p \ge P(t)$, it follows that the measure $\sum_{W \in \fL_V} \mu_W$ is finite.
In view of Proposition~\ref{p:conical conformal}, to complete we just need to show that $\sum_{W \in \fL_V} \mu_W$ is $(t, p)$\nobreakdash-sub-conformal for~$f$.
The proof of this fact is similar to what was done in~\cite[Proposition~B.2]{PrzRiv07}.
\end{proof}
\subsection{Equilibrium state}\label{ss:equilibrium}
The following are crucial estimates.
\begin{lemm}\label{l:tail estimate}
Suppose that the pressure function~$\pressure$ is finite on a neighborhood of~$(t, p) = (t_0, P(t_0))$ and that it vanishes at this point.
If~$\mu$ is the unique $(t_0, P(t_0))$\nobreakdash-conformal measure of~$F$, then the following properties hold.
\begin{enumerate}
\item[1.]
For every~$(t, p) \in \R^2$ and~$\gamma > 0$
$$ \int \left| t \ln |F'| + p m \right|^\gamma d\mu < + \infty. $$
\item[2.]
There is $\varepsilon_0 > 0$ such that for every sufficiently large integer~$n$ we have
$$ \sum_{\substack{W \text{ connected component of } D \\ m_W \ge n}} \mu(W) \le \exp(-\varepsilon_0 n). $$
\end{enumerate}
In particular
$$
\sum_{W \text{ connected component of } D} m_W \mu(W) < + \infty.
$$
\end{lemm}
We have stated part~1 for every~$(t, p) \in \R^2$, although we will only use it for~$(t, p)$ close to~$(t_0, p_0)$.
\begin{proof}
Since the function~$\pressure$ is finite on a neighborhood of $(t, p) = (t_0, P(t_0))$, there is $\varepsilon_0 > 0$ such that $\pressure(t_0 - \varepsilon_0, P(t_0) - \varepsilon_0) < + \infty$.
By ~\cite[Proposition~2.1.9]{MauUrb03} this implies that,
$$
\sum_{W \text{connected component of } D} \exp(-(P(t_0) - \varepsilon_0)m_W) \sup \{ |F'(z)|^{- (t_0 - \varepsilon_0)} \mid z \in W \}
<
+ \infty.
$$
As for each connected component~$W$ of~$D$ we have
$$ \mu(W) \le C_0 \exp(-P(t_0)m_W) \sup \{ |F'(z)|^{- t_0} \mid z \in W \}, $$
we obtain the conclusion of part~1 holds for each~$(t, p) \in \R^2$ and that
$$ C_1 \= \sum_{W \text{connected component of } D} \mu(W) \exp(\varepsilon_0 m_W) < + \infty. $$
So for each $n \ge 1$ we have
$$ \exp(\varepsilon_0 n) \sum_{\substack{W \text{ connected component of } D \\ m_W \ge n}} \mu(W) \le C_1. $$
This proves part~2 of the lemma.
\end{proof}

\subsubsection*{Existence}
It follows from standard considerations that~$F$ has an invariant measure~$\rho$ that is absolutely continuous with respect to the $(t_0, P(t_0))$\nobreakdash-conformal measure~$\mu$ of~$F$, and that the density of~$\rho$ with respect to~$\mu$ is bounded from below by a strictly positive constant almost everywhere.
This result can be found for example in~\cite[{\S}1]{gouezelthesis}, by observing that $F|_{J(F)}$ is a ``Gibbs-Markov map''.
For a proof in a setting closer to ours, but that only applies to the case when~$V$ is connected, see~\cite[{\S}6]{MauUrb03}.

The measure
$$
\hrho \= \sum_{W \text{connected component of } D} \sum_{j = 0}^{m_W - 1} f^j_* \rho|_W
$$
is easily seen to be invariant by~$f$ and part~2 of Lemma~\ref{l:tail estimate} implies that it is finite.
Furthermore this measure is absolutely continuous with respect to the $(t_0, P(t_0))$\nobreakdash-conformal measure~$\hmu$ of~$f$, and its density is bounded from below by a strictly positive constant on a subset of~$V$ of full measure with respect to $\mu = \hmu|_V$.
It follows from the locally eventually onto property of Julia sets that the density of~$\hrho$ with respect to~$\hmu$ is bounded from below by a strictly positive constant almost everywhere; see for example~\cite[\S8]{PrzRiv07} for details.
As~$\hmu$ is ergodic (Proposition~\ref{e:sub-conformal}) it follows that~$\hrho$ is also ergodic.

We will show now that the probability measure~$\trho$ proportional to~$\hrho$ is an equilibrium state of~$f$ for the potential $-t_0 \ln |f'|$.
We first observe that by part~1 of Lemma~\ref{l:tail estimate} and~\cite[Theorem~2.2.9]{MauUrb03} the measure~$\rho$ is an equilibrium state of~$F$ for the potential $-t_0 \ln |F'| - P(t_0) m$, see also Remark~\ref{r:using symbolic}.
That is, we have
$$
P(F, -t_0 \ln |F'| - P(t_0) m)
=
h_\rho(F) - \int t_0 \ln |F'| +  P(t_0) m \, d\rho,
$$
which is equal to~$0$ by hypothesis.
By the generalized Abramov's formula~\cite[Theorem~5.1]{Zwe05}, we have $h_\rho(F) = h_{\trho}(f) \hrho(\CC)$, and by definition of~$\hrho$ we have $\int m \, d\rho = \hrho(\CC)$.
We thus obtain,
\begin{multline*}
h_{\trho}(f)
=
(\hrho(\CC))^{-1} h_\rho(F)
=
(\hrho(\CC))^{-1} \int t_0 \ln |F'| +  P(t_0) m \, d\rho
\\ =
(\hrho(\CC))^{-1} t_0 \int \ln |f'| d\hrho + P(t_0)
=
t_0 \int \ln |f'| d\trho + P(t_0).
\end{multline*}
This shows that~$\trho$ is an equilibrium state of~$f$ for the potential $-t_0 \ln |f'|$.
\subsubsection*{Uniqueness}
In view of~\cite[Theorem~8]{Dob0804}, we just need to show that the Lyapunov exponent of each equilibrium state of~$f$ for the potential $-t_0 \ln |f'|$ is strictly positive; see also~\cite{Led84}.

Let~$\trho'$ be an equilibrium state of~$f$ for the potential $-t_0 \ln |f'|$. 
If~$f$ satisfies the \TCE{} then it follows that the Lyapunov exponent of~$\trho'$ is strictly positive, as in this case we have $\chiinf > 0$.
Otherwise we have $\chiinf = 0$, and then $P(t_0) > 0$ by Proposition~\ref{p:asymptotes}.
It thus follows that~$h_{\trho'}(f) > 0$, and therefore that the Lyapunov exponent of~$\trho'$ is strictly positive by Ruelle's inequality.
\subsubsection*{Statistical properties}
When~$F$ satisfies the conclusions of Lemma~\ref{l:mixingness}, the statistical properties of~$\trho$ can be deduced from the tail estimate given by part~2 of Lemma~\ref{l:tail estimate}, using Young's results in~\cite{You99}.
In the case when there is only one critical point in the Julia set one can apply these results directly, and in the general case one needs to consider the first return map of~$F$ to the set~$V^{\widetilde{c}}$, where $\widetilde{c}$ is the critical point given by the conclusion of~Lemma~\ref{l:mixingness}, as it was done in~\cite[\S8.2]{PrzRiv07}.
In the general case one could also apply directly the generalization of Young's result given in~\cite[Th{\'e}or{\`e}me~2.3.6 and Remarque~2.3.7]{gouezelthesis}.
We omit the standard details.
\subsection{Analyticity of the pressure function}\label{ss:lifting analyticity}
By hypothesis for each~$t$ close to~$t_0$ we have $\pressure(t, P(t)) = 0$.
Since the function~$\pressure$ is real analytic on a neighborhood of $(t_0, P(t_0))$ (Lemma~\ref{l:pressure}), by the implicit function theorem it is enough to check that $\frac{\partial}{\partial p} \pressure|_{(t_0, P(t_0))} \neq 0$.
By part~1 of Lemma~\ref{l:tail estimate} and~\cite[Proposition~2.6.13]{MauUrb03} this last number is equal to the integral of the (strictly negative) function~$- m$ against~$\rho$, see also~Remark~\ref{r:using symbolic}.
It is therefore strictly negative.

\section{Whitney decomposition of a pull-back}\label{s:Whitney decomposition}
The purpose of this section is to introduce a Whitney type decomposition of a given pull-back of a pleasant couple.
It is used to prove the key estimates in the next section.
\subsection{Dyadic squares}\label{ss:dyadic squares}
Fix a square root~$i$ of~$-1$ in~$\C$ and identify~$\C$ with $\R \oplus i \R$.
For integers $j, k$ and $\ell$, the set
$$ \left\{ x + i y \mid x \in \left[ \tfrac{j}{2^\ell}, \tfrac{j + 1}{2^\ell} \right], y \in \left[ \tfrac{k}{2^\ell}, \tfrac{k + 1}{2^\ell} \right] \right\}, $$
will be called \emph{dyadic square}.
Note that two dyadic squares are either nested or have disjoint interiors.
We define \emph{a quarter} of a dyadic square~$Q$ as one of the four dyadic squares contained in~$Q$ and whose side length is one half of that of~$Q$.

Given a dyadic square~$Q$, denote by~$\hQ$ the open square having the same center as~$Q$, sides parallel to that of~$Q$, and length twice as that of~$Q$.
Note in particular that for each dyadic square~$Q$ the set~$\hQ \setminus Q$ is an annulus whose modulus is independent of~$Q$; we denote this number by~$\mconst_1$.

\subsection{Primitive squares}\label{ss:primitive squares}
Let~$f$ be a rational map of degree at least two.
We fix $r_1 > 0$ sufficiently small so that for each critical value~$v$ of~$f$ in the Julia set of~$f$ there is a univalent map $\varphi_v : B(v, 9r_1) \to \C$ whose distortion is bounded by~2.

We say that a subset~$Q$ of~$\CC$ is a \emph{primitive square}, if there is~$v \in \CVJ$ such that~$Q$ is contained in the domain of~$\varphi_v$, such that~$\varphi_v(Q)$ is a dyadic square, and such that $\widehat{\varphi_v(Q)}$ is contained in the image of~$\varphi_v$.
In this case we put $v(Q) \= v$ and $\hQ \= \varphi_v^{-1}\left( \widehat{\varphi_v(Q)} \right)$.
We say that a primitive square~$Q_0$ is \emph{a quarter} of a primitive square~$Q$, if~$Q_0 \subset Q$ and if~$\varphi_{v(Q)}(Q_0)$ is a quarter of~$\varphi_{v(Q)}(Q)$.
Note that each primitive square has precisely four quarters.
Furthermore, each primitive square~$Q$ contained in~$B(\CV, r_1)$ is contained in a primitive square~$Q'$ such that~$Q$ is a quarter of~$Q'$.
\begin{defi}
Fix~$\Delta \in (0, r_1)$.
The \emph{Whitney decomposition} associated to (the complement of) a subset~$F$ of~$\CC$ is the collection~$\sW(F)$ of all those primitive squares~$Q$ such that $\diam(Q) < \Delta$, $\hQ \cap F = \emptyset$, and that are maximal with these properties.  
\end{defi}
By definition two distinct elements of~$\sW(F)$ have disjoint interiors, and each point in~$B(\CVJ, 9r_1) \setminus F$ is contained in an element of~$\sW(F)$.
\begin{lemm}\label{l:Whitney of finite}
Let~$\Delta \in (0, r_1)$, and let~$F$ be a finite subset of~$\CC$.
Then the following properties hold.
\begin{enumerate}
\item[1.]
Let~$Q_0$ be a primitive square contained in~$B(\CVJ, r_1)$ and such that $\diam(Q_0) \le \Delta$.
Then either~$Q_0$ is contained in an element of~$\sW(F)$, or it contains an element~$Q$ of~$\sW(F)$ such that
$$ \diam(Q) \ge \tfrac{1}{4} (2 + 3 \sqrt{\# F})^{-1} \diam(Q_0). $$
\item[2.]
For each~$n \ge 2$ the number of those~$Q \in \sW(F)$ contained in~$B(\CVJ, r_1)$ and such that $\diam(Q) \in [2^{ -(n + 1)} \Delta, 2^{-n} \Delta]$ is less than~$2599 (\# F)$.
\end{enumerate}
\end{lemm}
\begin{proof}
\

\partn{1}
Let $n \ge 2$ be the least integer such that~$(2^n - 2)^2 > 9 (\# F)$, so that $2^n \le 2 \left( 2 + 3 \sqrt{\# F} \right)$.
Put $Q_0' \= \varphi_{v(Q)}(Q_0)$ and denote by~$\ell_0$ the side length of~$Q_0'$.
For each element~$a$ of~$F$ in~$Q_0$ choose a dyadic square~$Q_a$ whose side length equal to~$2^{-n} \ell_0$ and that contains~$\varphi_{v(Q)}(a)$.
As there are $(2^n - 2)^2$ squares of side length equal to~$2^{-n} \ell_0$ contained in the interior of~$Q_0'$, and at most $9 (\#F) < (2^n - 2)^2$ of them intersect one of the squares~$\bigcup_{a \in F} Q_a$, we conclude that there is at least one square~$Q'$ of side length equal to~$2^{-n} \ell_0$ that is contained in the interior~$Q_0'$ and such that $\varphi_{v(Q)}^{-1} \left( \widehat{Q'} \right)$ is disjoint from~$F$.
It follows that the primitive square $Q \= \varphi_{v(Q)}^{-1}(Q')$ is contained in an element of~$\sW(F)$.
As,
$$ \diam(Q) \ge \tfrac{1}{2} 2^{-n} \diam(Q_0) \ge \tfrac{1}{4} (2 + 3 \# \sqrt{F})^{-1} \diam(Q_0), $$
the desired assertion follows.

\partn{2}
Let~$Q$ be an element of~$\sW(F)$ contained in~$B(\CVJ, r_1)$, and let~$Q'$ be a primitive square such that~$Q$ is a quarter of~$Q'$.
Then either~$\diam(Q') > \Delta$ or $\widehat{Q'}$ intersects~$F$.
So, if $\diam(Q) \le \tfrac{1}{4} \Delta$, then there is~$a \in F$ contained in $\widehat{Q'}$, and therefore $\diam(Q) \ge \tfrac{1}{4} \dist(Q, a)$.
So, if we let~$n \ge 2$ be an integer such that $\diam(Q) \in [2^{- (n + 1)} \Delta, 2^{-n} \Delta]$, then $Q \subset B(a, 5 \cdot 2^{-n} \Delta)$.
Since the area of~$Q$ is greater than or equal to~$\tfrac{1}{8} \diam(Q)^2 \ge \tfrac{1}{32} 4^{-n} \Delta^2$ and the area of $B(a, 5 \cdot 2^{-n} \Delta)$ is less than $25 \pi 4^{-n} \Delta^2$, we conclude that there are at most $ 25 \cdot 32 \pi (\#F) < 2599 (\# F)$ elements~$Q$ of~$\sW(F)$ satisfying $\diam(Q) \in [2^{- (n + 1)} \Delta, 2^{-n} \Delta]$.
\end{proof}

\subsection{Univalent squares}\label{ss:univalent squares}
For an integer~$n \ge 0$ we will say that a subset~$Q$ of~$\CC$ is \emph{a univalent square of order~$n$}, if there is a primitive square~$Q'$ such that~$Q$ is a connected component of~$f^{- (n + 1)}(Q')$, and such that~$f^{n + 1}$ is univalent on the connected component of~$f^{- (n + 1)}(\widehat{Q'})$ containing~$Q$.
In this case we denote this last set by~$\hQ$, and note that $\hQ \setminus Q$ is an annulus of modulus equal to~$\mconst_1$.
It thus follows that there is a constant $K_0 > 1$ such that for every univalent square~$Q$ of order~$n$ and every~$j = 1, \ldots, n + 1$, the distortion of~$f^{j}$ on~$Q$ is bounded by~$K_0$.

Let $(\hV, V)$ be a pleasant couple for~$f$ such that $f(\hV) \subset B(\CVJ, r_1)$.
For a pull-back~$\badp$ of~$\hV$, denote by~$\ell(\badp)$ the number of those~$j \in \{0, \ldots, m_{\badp} \}$ such that $f^j(\badp) \subset \hV$.
Moreover, let~$\sW (\badp)$ be the collection of all those univalent squares~$Q$ that are of order~$m_{\badp}$, such that~$\hQ \subset \badp$, such that~$f^{m_{\badp}}(Q)$ intersects~$V$, and that are maximal with these properties.
Note that for~$Q \in \sW(\badp)$ we have $v(Q) = f(c(\badp))$.
By definition every pair of distinct elements of~$\sW(\badp)$ have disjoint interiors.
On the other hand, every point in~$f^{m_{\badp}}|_{\badp}^{-1} \left( V^{c(\badp)} \right) \setminus \Crit(f^{m_{\badp} + 1})$ is contained in a an element of~$\sW(\badp)$, and for each~$Q \in \sW(\badp)$ the set~$\hQ$ is disjoint from~$\Crit(f^{m_{\badp} + 1})$.

\begin{prop}\label{p:backward Whitney}
Let~$f$ be a rational map of degree at least two and let~$(\hV, V)$ be a pleasant couple for~$f$.
Then there is a constant~$C_0 > 0$ such that for every $\xi \in (0, 1)$ the number of those~$Q \in \sW(\badp)$ such that
$$ \diam \left( f^{m_{\badp} + 1}(Q) \right)
\ge
\xi \diam(f(\hV^{c(\badp)})) $$
is less than
$$ 2600 \deg(f)^{\ell(\badp)} \left (C_0 + \tfrac{1}{2} \ell(\badp) \log_2 \ell(\badp) + \ell(\badp) \log_2 (\xi^{-1}) \right). $$
\end{prop}
\begin{proof}
Put~$c = c(\badp)$ and $v = f(c)$, and let~$\xi_0 \in (0, 1)$ be sufficiently small so that for each~$z \in V^c$ the connected component of~$f^{-1}\left(B ( f(z), \xi_0 \diam(f( \hV^c))) \right)$ containing~$z$ is contained in~$\hV^c$.
Put~$F = f^{m_{\badp} + 1} \left( \badp \cap \Crit \left(f^{m_{\badp} + 1} \right) \right)$, $\Delta \= \xi_0 \diam(f( \hV^c ))$ and consider the Whitney decomposition~$\sW(F)$, as defined in~\S\ref{ss:primitive squares}.
Note that~$\# F \le \ell(\badp)$.

\partn{1}
We prove first that for every $Q \in \sW(\badp)$ the primitive square $f^{m_{\badp} + 1}(Q) \subset f( \hV^c ) \subset B(v, r_1)$ contains an element~$Q'$ of~$\sW(F)$ such that
$$ \diam(Q') \ge \left( 80 \sqrt{\# \ell(\badp)} \right)^{-1} \xi_0 \diam \left( f^{m_{\badp} + 1}(Q) \right).$$
Let~$Q_0$ be a primitive square contained in~$f^{m_{\badp} + 1}(Q)$ such that
$$ \tfrac{1}{4} \xi_0 \diam(f^{m_{\badp} + 1}(Q)) \le \diam(Q_0) \le \Delta. $$
By part~1 of Lemma~\ref{l:Whitney of finite} there is an element~$Q'$ of~$\sW(F)$ that either contains~$Q_0$, or that it is contained in~$Q_0$ and
\begin{multline*}
\diam(Q')
\ge
\tfrac{1}{4} \left( 2 + 3 \sqrt{\# F} \right)^{-1} \diam(Q_0)
\\ \ge
\tfrac{1}{16} \left( 2 + 3 \sqrt{\# F} \right)^{-1} \xi_0 \diam \left( f^{m_W + 1}(Q) \right). 
\end{multline*}
As $\# F \le \ell(\badp)$ and~$\ell(\badp) \ge 1$, we just need to show that~$Q'$ is in fact contained in~$f^{m_{\badp} + 1}(Q)$.
Suppose by contradiction that this is not the case.
Then it follows that~$Q'$ contains~$f^{m_{\badp} + 1}(Q)$ strictly.
Let~$\widetilde{Q'}$ be the connected component of~$f^{- (m_{\badp} + 1)}(Q')$ containing~$Q$.
By definition of~$\xi_0$ we have that $f^{m_{\badp}}(\widetilde{Q'})$ is contained in~$\hV^c$, so~$\widehat{\widetilde{Q'}}$ is contained in~$\badp$.
On the other hand~$f^{m_{\badp}}(\widetilde{Q'})$ intersects~$V^c$, because it contains~$f^{m_{\badp}}(Q)$ and this set intersects~$V^c$.
As by definition of~$\sW(F)$ the set~$\widehat{Q'}$ is disjoint from~$F$, it follows that~$f^{m_{\badp}  + 1}$ is univalent on~$\widehat{\widetilde{Q'}}$.
Thus, by definition of~$\sW(\badp)$, the univalent square~$\widetilde{Q'}$ is contained in an element of~$\sW(\badp)$.
But $Q \in \sW(\badp)$ is strictly contained in~$\widetilde{Q'}$, so we get a contradiction.
This shows that~$Q'$ is in fact contained in~$f^{m_{\badp} + 1}(Q)$ and completes the proof of the assertion.

\partn{2}
For each~$Q \in \sW(\badp)$ choose an element~$Q'$ of~$\sW(F)$ satisfying the property described in part~1.
Note that for each~$Q_0' \in \sW(F)$ the number of those~$Q \in \sW(\badp)$ such that $Q' = Q_0'$ is less than or equal to~$\deg(f)^{\ell(\badp)}$. 
As the area of a primitive square~$Q'$ is greater than or equal to~$\tfrac{1}{8} \diam(Q')^2$, it follows that for each~$\xi \in (0, 1)$ the number of those~$Q \in \sW(\badp)$ satisfying $ \diam (Q') \ge \xi \diam ( f( \hV^c ) )$ is less than or equal to~$8 \pi \xi^{-2} \deg(f)^{\ell(\badp)}$.

Let~$\xi \in (0, \tfrac{1}{4} \xi_0)$ be given and let~$n_0$ be the least integer $n \ge 2$ such that~$\xi \ge 2^{-n} 80 \sqrt{\ell(\badp)}$, so that $\xi < 2^{-(n_0 - 1)} 80 \sqrt{\ell(\badp)}$.
If~$Q \in \sW(\badp)$ is such that $\diam \left( f^{m_{\badp} + 1}(Q) \right) \ge \xi \diam(\hV^c)$, then we have
$$ \diam(Q')
\ge
\left( 80 \sqrt{\ell(\badp)} \right)^{-1} \xi_0 \diam \left( f^{m_{\badp} + 1}(Q) \right)
\ge 2^{-n_0} \xi_0 \diam ( \hV^c). $$
So part~2 of Lemma~\ref{l:Whitney of finite} implies that for each~$n \ge 2$ the number of those~$Q \in \sW(\badp)$ such that
$$ \diam(Q') \in \left[ 2^{- (n + 1)} \xi_0 \diam ( f( \hV^c )), 2^{-n} \xi_0 \diam ( f( \hV^c )) \right], $$
is less than~$2599 (\# F) \deg(f)^{\ell(\badp)} \le 2599 \ell(\badp)\deg(f)^{\ell(\badp)}$.
So we conclude that the number of those~$Q \in \sW(\badp)$ such that $\diam \left( f^{m_{\badp} + 1}(Q) \right) \ge \xi \diam(\hV^c)$ is less than
\begin{multline*}
\deg(f)^{\ell(\badp)} \left( 8 \pi (\tfrac{1}{4} \xi_0)^{-2}  + (n_0 - 2) 2599 \ell(\badp) \right)
\\ \le
\deg(f)^{\ell(\badp)} \left( 8 \pi (\tfrac{1}{4} \xi_0)^{-2} + 2599 \ell(\badp) \left(\log_2 (\xi^{-1}) + \log_2(80) + \tfrac{1}{2} \log_2 (\ell(\badp)) \right) \right).
\end{multline*}
This completes the proof of the lemma.
\end{proof}

\section{The contribution of a pull-back}\label{s:pull-back contribution}
Fix a rational map~$f$ of degree at least two, and a pleasant couple $(\hV, V)$ for~$f$.
Recall that~$\fL_V$ is the collection of all the connected components of~$\CC \setminus K(V)$ and that for a pull-back~$\badp$ of~$\hV$ we denote by~$\fD_{\badp}$ the collection of all the pull-backs~$W$ of~$V$ that are contained in~$\badp$, such that~$f^{m_{\badp}}(W) \subset V$, such that~$f^{m_{\badp} + 1}$ is univalent on~$\hW$ and such that $f^{m_{\badp} + 1}(W) \in \fL_V$; see~\S\ref{ss:bad pull-backs}.
Furthermore, we denote by~$\ell(\badp)$ the number of those~$j \in \{ 0, \ldots, m_{\badp} \}$ such that $f^j(\badp) \subset \hV$.

The purpose of this section is to prove the following.
\begin{prop}[Key estimates]\label{p:pull-back contribution}
Let~$f$ be a rational map of degree at least two that is expanding away from critical points.
Then for each sufficiently small pleasant couple~$(\hV, V)$ for~$f$ the following properties hold.
\begin{enumerate}
\item[1.]
For every $t_0 \in \R$, and every~$(t, p) \in \R^2$ sufficiently close to~$(t_0,  P(t_0))$, we have
\begin{equation}\label{e:pressure outside}
\sum_{W \in \fL_V} \exp( - p m_W) \diam(W)^{t} < + \infty.
\end{equation}
\item[2.]
Let~$t, p \in \R$ be such that~\eqref{e:pressure outside} holds and such that
$$p > \max \{ - t \chiinf, - t \chisup \}. $$
Then for every $\varepsilon > 0$ such that
$$ |t| \varepsilon < p - \max \{ - t \chiinf, - t \chisup \}, $$
there is a constant~$C_1 > 0$ such that for each pull-back~$\badp$ of~$\hV$ we have
\begin{multline*}
\sum_{W \in \fD_{\badp}} \exp( - p m_W) \diam(W)^{t}
\\ \le
C_1 (\deg(f) + 1)^{\ell(\badp)}\exp \left( - m_{\badp} (p - \max \{ - t \chiinf, - t \chisup \} - |t| \varepsilon) \right).
\end{multline*}
\end{enumerate}
\end{prop}
The prove this proposition we start with the following lemma.
\begin{lemm}\label{l:pressure outside}
Let~$f$ be a rational map that is expanding away from critical points.
Then for every compact and forward invariant subset $K$ of the Julia set of~$f$ that is disjoint from the critical points of~$f$ and every $t > 0$ we have
$$
P(f|_K, -t \ln |f'|) < P(t).
$$
\end{lemm}
\begin{proof}
By hypothesis~$f$ is uniformly expanding on~$K$.
Enlarging $K$ if necessary we may assume that the restriction of~$f$ to~$K$ admits a Markov partition, see \cite[Theorem~3.5.2 and Remark~3.5.3]{PUbook},\footnote{An analogous result in the case of diffeomorphisms is shown in~\cite{Fis06}.} so that there is at least one equilibrium state~$\mu$ for $f|_K$ with potential $-t \ln |f'|$.

We enlarge $K$ with more cylinders to obtain a compact forward
invariant subset $K'$ of $J(f)$, so that $f$ restricted to $K'$
admits a Markov partition and so that the relative interior of $K$
in $K'$ is empty. It follows that $\mu$ cannot be an equilibrium
measure for $f|_{K'}$ for the potential $-t \ln |f'|$, so we have
$$
P(f|_K, - t \ln |f'|) = h_\mu(f) - t \int_K \ln |f'| d \mu <
P(f|_{K'}, - t \ln |f'|) \le P(t).
$$
\end{proof}

To prove Proposition~\ref{p:pull-back contribution}, let~$f$ be a rational map of degree at least two, and let $(\hV, V)$ be a pleasant couple for~$f$.
We will define a constant~$r_0 > 0$ as follows.
If~$\chiinf = 0$ we put~$r_0 = \dist (\partial V, \CJ)$.
Suppose that~$\chiinf > 0$.
Then by~\cite[Main Theorem]{PrzRivSmi03} there exists~$r_0' > 0$ such that for every~$z_0$ in~$J(f)$, every~$\varepsilon > 0$, every sufficiently large integer~$n$, and every connected component~$W$ of~$f^{-n}(B(z_0, r_0'))$, we have
$$ \diam(W) \le \exp (-n(\chiinf  - \varepsilon)). $$
Then we put~$r_0 = \min \{ r_0', \dist (\partial V, \CJ) \}$.

Given a subset~$Q$ of~$\CC$ we define~$n_Q \in \{ 0, 1, \ldots, + \infty \}$ as follows.
If there are infinitely many integers~$n$ such that~$\diam(f^n(Q)) < r_0$, then we put~$n_Q = + \infty$.
Otherwise we let~$n_Q$ be the largest integer~$n \ge 0$ such that~$\diam(f^n(Q)) < r_0$.
\begin{lemm}\label{l:size versus time}
Let~$f$ be a rational map of degree at least two.
Then for every $\varepsilon > 0$ there is a constant~$C(\varepsilon) > 1$ such that for each connected subset~$Q$ of~$\CC$ that intersects the Julia set of~$f$ we have, 
\begin{equation*}\label{e:size versus time}
C(\varepsilon)^{-1} \exp(-n_Q (\chisup + \varepsilon))
\le
\diam(Q)
\le
C(\varepsilon) \exp(-n_Q (\chiinf - \varepsilon))
\end{equation*}
\end{lemm}
\begin{proof}
The inequality on the right holds trivially when~$\chiinf = 0$, and when $\chiinf > 0$ it is given by the definition of~$r_0 > 0$.
The inequality on the left is a direct consequence of part~2 of Proposition~\ref{p:individual pressure}.
\end{proof}

\begin{proof}[Proof of Proposition~\ref{p:pull-back contribution}]
Let~$r_1 > 0$ be as in the definition of primitive squares in~\S\ref{ss:primitive squares}, and let~$(\hV, V)$ be a pleasant couple for~$f$ such that~$f(\hV) \subset B(\CVJ, r_1)$.
Furthermore, let~$A_1 > 0$ and~$K_1 > 1$ be given by Koebe Distortion Theorem in such a way that for each pull-back~$W$ of~$V$ such that~$f^{m_W}$ is univalent on~$\hW$ we have $\diam(W) \le A_1 \dist(W, \partial \hW)$, and such that for each $j = 1, \ldots, m_W$ the distortion of~$f^{j}$ on~$W$ is bounded by~$K_1$.

\partn{1}
Note that it is enough to show that there are~$t < t_0$ and~$p < P(t_0)$ for which~\eqref{e:pressure outside} holds.

Let $V'$ be a sufficiently small neighborhood of $\CJ$ contained in~$V$, so that for each $c \in \CJ$ the set
$$
K' = \{ z \in J(f) \mid \text{ for every $n \ge 0$, $f^n(z) \not \in V'$} \}
$$ 
intersects $V^c$.
By Lemma~\ref{l:pressure outside} we have $P(f|_{K'}, -t_0 \ln|f'|) < P(t_0)$.
Let~$t < t_0$ and~$p < P(t_0)$ be sufficiently close to~$t_0$ and~$P(t_0)$, respectively, so that $p > P(f|_{K'}, -t \ln|f'|)$.

For each $c \in \CJ$ choose a point~$z(c)$ in $K' \cap V^c$.
Given a univalent pull-back~$W$ of~$V$ let~$z_W$ be the unique point in $f^{-m_W}(z(c(W)))$ contained in~$W$.
Note that when~$W \in \fL_V$ we have~$z_W \in K'$.
On the other hand, there is a distortion constant $C > 0$ such that for each pull-back~$W$ of~$V$ such that~$f^{m_W}$ maps a neighborhood of~$W$ univalently onto a component of~$\hV^c$, we have $\diam(W) \le C |(f^{m_W})'(z_W)|^{-1}$.

Since by hypothesis the restriction of~$f$ to~$K'$ is uniformly expanding, we
have
\begin{multline*}
\limsup_{n \to + \infty} \tfrac{1}{n} \ln \sum_{\substack{W \in \fL_V \\ m_W = n}} |(f^{n})'(z_W)|^{-t}
\\ \le
\limsup_{n \to + \infty} \tfrac{1}{n} \ln \sum_{c \in \CJ} \sum_{z \in K' \cap f^{-n} (z(c))} |(f^n)'(z)|^{- t} 
\\ \le
P(f|_{K'}, -t \ln |f'|),
\end{multline*}
hence
\begin{multline*}
C_2 \= \sum_{W \in \fL_V} \exp(- m_W p) \diam(W)^{t}
\\ \le
C^{|t|} \sum_{W \in \fL_V} \exp(- m_W p) |(f^{m_W})'(z_W)|^{-t}
<
+ \infty.
\end{multline*}

\partn{2.1}
Put
$$ C_3 \= \min \{ \dist(z(c), \partial V^c) / \diam(V^c) \mid c \in \CJ  \} $$
and observe that for each pull-back~$W'$ of~$V$ such that~$\widehat{W'}$ is a univalent pull-back of~$\hV$, we have $\dist(z_{W'}, \partial W') \ge C_3 K_1^{-1} \diam(W')$.
We will show that for each pull-back~$\badp$ of~$\hV$, for each~$Q \in \sW(\badp)$, and each $W \in \fD_{\badp}$ such that $z_W \in Q$, we have
$$ \diam(f^{m_{\badp} + 1}(W)) \le 8 C_3^{-1} K_1 \diam \left( f^{m_{\badp} + 1}(Q) \right). $$
Put~$Q' = f^{m_{\badp} + 1}(Q)$ and $W' = f^{m_{\badp} + 1}(W)$, and suppose by contradiction that~$\diam(W') > 8 C_3^{-1} K_1 \diam(Q')$.
Observe that~$Q'$ is a primitive square contained in~$B(f(c(\badp)), r_1)$ and that $W' \in \fL_V$.
So there is a primitive square~$Q_0'$ such that~$Q'$ is a quarter of~$Q_0'$.
We have
$$ \diam ( \widehat{Q'}_0 ) \le 8 \diam(Q') < C_3 K_1^{-1} \diam(W') \le \dist(z_{W'}, \partial W'). $$
Since by hypothesis $z_W \in Q$, we have $z_{W'} = f^{m_{\badp} + 1}(z_W) \in Q' \subset \widehat{Q'}_0$, so the last inequality implies that $\widehat{Q'}_0 \subset~W'$.
But~$f^{m_{\badp} + 1}$ is univalent on~$W$, so the connected component~$Q_0$ of $f^{- (m_{\badp} + 1)} \left( Q_0' \right)$ containing~$Q$ is a univalent square of order~$m_{\badp}$ satisfying~$\widehat{Q}_0 \subset \badp$, that contains~$Q$ strictly.
This contradicts the hypothesis that $Q \in \sW(\badp)$.

\partn{2.2}
We will now show that there is a constant~$C_4 > 0$ such that for each pull-back~$\badp$ of~$\hV$ and each~$Q \in \sW(\badp)$ we have 
\begin{equation}\label{e:square contribution}
\sum_{\substack{W \in \fD_{\badp} \\ z_W \in Q}} \exp(- p m_{W}) \diam(W)^{t}
\le
C_4 \exp(- p n_Q) \diam(Q)^{t}.
\end{equation}

Let~$\badp$ be a pull-back of~$\hV$ and let~$Q \in \sW(\badp)$.
Put $Q' = f^{m_{\badp} + 1}(Q)$, and let~$B$ be a ball whose center belongs to~$Q'$ and of radius equal to~$(8C_3^{-1}K_1 + 1)\diam(Q')$.
By part~2.1, for each~$W \in \fD_{\badp}$ such that $z_W \in Q$ we have $f^{m_{\badp} + 1}(W) \subset B$.
Since the distortion of~$f^{m_{\badp} + 1}$ is bounded by~$K_0$ on~$Q$, and by~$K_1$ on each element of~$\fD_{\badp}$, we obtain,
\begin{multline*}
\sum_{\substack{W \in \fD_{\badp} \\ z_W \in Q}} \exp(- p m_{W}) \diam(W)^{t}
\\ \le
\exp(- p (m_{\badp} + 1)) (K_0K_1)^{|t|} \left( \frac{\diam(Q)}{\diam(Q')} \right)^{t} 
\cdot \\ \cdot
\sum_{\substack{W' \in \fL_V \\ W' \subset B}} \exp(- p m_{W'}) \diam(W')^{t}.
\end{multline*}

If there is no~$W \in \fD_{\badp}$ such that~$z_W \in Q$, then there is nothing to prove.
So we assume that there is an element~$W_0$ of~$\fD_{\badp}$ such that~$z_{W_0} \in Q$.
Then~$Q'$, and hence~$B$, intersects~$K'$, as it contains the point~$z_{f^{m_{\badp} + 1}(W_0)}$.
Since by hypothesis the restriction of~$f$ to~$K'$ is uniformly expanding, there is~$n_0 \ge 0$ independent of~$\badp$, such that $n_{Q'} \le n_B + n_0$ and such that there is an integer~$n_B' \ge 0$ satisfying~$|n_B' - n_B| \le n_0$, such that ~$f^{n_B'}$ is univalent on~$B$ and has distortion bounded by~2 on this set.
We have $|n_{Q'} - n_B'| \le 2n_0$, so there is a constant~$C_5 > 0$ independent of~$B$ such that $\diam(f^{n_B'}(B)) > C_5$.
So, if we put
$$ C_6 \= \exp(|p| 2n_0) \left( 2C_5^{- 1} 2(8C_3^{-1}K_1 + 1) \right)^{|t|}C_2,
$$
then we have
\begin{multline*}
\sum_{\substack{W' \in \fL_V \\ W' \subset B}} \exp( - p m_{W'}) \diam(W')^{t}
\\ \le
\exp(- p n_B') 2^{|t|} \left( \frac{\diam(B)}{\diam(f^{n_B'}(B))} \right)^{t}
\cdot
\sum_{\substack{W'' \in \fL_V \\ W'' \subset f^{n_B'}(B)}} \exp( - p m_{W''}) \diam(W'')^{t}
\\ \le
\exp(-p n_{Q'}) \exp ( - |p| 2n_0) (2C_5^{-1})^{|t|} \diam(B)^t C_2 
\\ \le
C_6 \exp( - p n_{Q'}) \diam(Q')^{t}.
\end{multline*}

Inequality~\eqref{e:square contribution} with constant~$C_4 \= C_6 (K_0K_1)^{|t|}$, is then a direct consequence of the last two displayed (chains of) inequalities.

\partn{2.3}
We will now complete the proof of the proposition.
For each~$Q \in \sW(\badp)$ put~$Q' \= f^{m_{\badp} + 1}(Q)$.
Let~$Q \in \sW(\badp)$ be such that there is~$W \in \fD_{\badp}$ satisfying~$z_W \in Q$.
As this last point is in the Julia set of~$f$, by Lemma~\ref{l:size versus time} we have
$$ \diam(Q)^{t}
\le
C(\varepsilon)^{|t|} \exp \left( n_Q (\max \{ - t \chisup, - t \chiinf \} + |t|\varepsilon) \right). $$
Since the elements of~$\sW(\badp)$ cover $\badp \setminus \Crit \left( f^{m_{\badp} + 1} \right)$, if we put
$$\gamma
\=
\exp ( - p + \max \{ - t \chisup, - t \chiinf \} + |t| \varepsilon ) \in (0, 1), $$
then by summing over~$Q \in \sW(\badp)$ in~\eqref{e:square contribution} we obtain
\begin{multline*}
\sum_{W \in \fD_{\badp}} \exp(- p m_{W}) \diam(W)^{t}
\\ \le
C_4 C(\varepsilon) \sum_{\substack{Q \in \sW(\badp) \\ Q \cap J(f) \neq \emptyset}} \gamma^{n_Q}
=
C_4 C(\varepsilon) \gamma^{m_{\badp} + 1} \sum_{\substack{Q \in \sW(\badp) \\ Q \cap J(f) \neq \emptyset}} \gamma^{n_{Q'}}.
\end{multline*}

To estimate this last number, observe that by Lemma~\ref{l:size versus time}, for each~$Q \in \sW(\badp)$ intersecting the Julia set of~$f$ we have
$$ \diam(Q') \ge C(\varepsilon)^{-1} \exp( - n_{Q'} (\chisup + \varepsilon)). $$
So, if we put $\tgamma = \gamma^{\tfrac{\ln 2}{\chisup + \varepsilon}}$, $C_7 = \tgamma^{- \log_2 C(\varepsilon) - \log_2 \diam(\hV^{c(\badp)})}$ and for each $Q \in \sW(\badp)$ we put $\xi(Q') = \diam(Q') / \diam(\hV^{c(\badp)})$, then we have $\gamma^{n_{Q'}} \le C_7 \tgamma^{- \log_2 \xi(Q')}$.
So Proposition~\ref{p:backward Whitney} implies that,
\begin{multline*}
\sum_{ \substack{Q \in \sW(\badp) \\ Q \cap J(f) \neq \emptyset}} \gamma^{n_{Q'}}
\le
C_7 \sum_{ \substack{Q \in \sW(\badp) \\ Q \cap J(f) \neq \emptyset}} \tgamma^{- \log_2 \xi(Q')}
\le \\ \le
2600 C_7 \deg(f)^{\ell(\badp)} \left( C_0 + \tfrac{1}{2} \ell(\badp) \log_2 \ell(\badp) + \ell(\badp) \sum_{n = 0}^{+ \infty} \tgamma^n \right).
\end{multline*}
This completes the proof of the proposition.
\end{proof}

\section{Proof of the Main Theorem}\label{s:proof of nice thermodynamics}
The purpose of this section is to prove the following version of the Main Theorem for pleasant couples.
Recall that each nice couple is pleasant and satisfies property~(*), see~\S\ref{ss:two variable pressure}.
\begin{theoalph}\label{t:nice thermodynamics}
Let~$f$ be a rational map of degree at least two that is expanding away from critical points, and that has arbitrarily small pleasant couples having property~(*).
Then following properties hold.
\begin{description}
\item[Analyticity of the pressure function]
The pressure function of~$f$ is real analytic on $(\tneg, \tpos)$, and linear with slope~$- \chisup(f)$ (resp. $-\chiinf(f)$) on $(- \infty, \tneg]$ (resp. $[\tpos, + \infty)$).
\item[Equilibrium states]
For each $t_0 \in (\tneg, \tpos)$ there is a unique equilibrium state of~$f$ for the potential~$-t_0 \ln|f'|$.
Furthermore this measure is ergodic and mixing.
\end{description}
\end{theoalph}

Throughout the rest of this section we fix a rational map~$f$ and~$t_0 \in (\tneg, \tpos)$ as in the statement of the theorem.
Recall that by~Proposition~\ref{p:asymptotes} we have $P(t_0) > \max \{ -t_0 \chiinf, - t_0 \chisup \}$.
Put
$$ \gamma_0
\=
\exp \left( - \tfrac{1}{2} (P(t_0) - \max \{ -t_0 \chiinf, - t_0 \chisup \}) \right) \in (0, 1), $$
and choose $L \ge 0$ sufficiently large so that
\begin{equation}\label{e:constants choice}
(2L \deg(f) (\deg(f) + 1) \#(\CJ))^{1/L} \gamma_0 < 1.
\end{equation}
Let~$(\hV, V)$ be a pleasant couple for~$f$ that is sufficiently small so that for each $\ell \in \{ 1, \ldots, L \}$ the set~$f^{\ell}(\CJ)$ is disjoint from~$\hV$ (recall that our standing convention is that no critical point of~$f$ in its Julia set is mapped to a critical point under forward iteration.)
We assume furthermore that~$(\hV, V)$ has property~(*).
By Lemma~\ref{l:bad pull-backs} it follows that for each integer~$n \ge 1$ and each~$z_0 \in V$ the number of bad iterated pre-images of~$z_0$ of order~$n$ is at most
$$ (2L \deg(f) \#(\CJ))^{n/L} $$
and that the number of bad pull-backs of~$\hV$ of order~$n$ is at most
$$ \# (\CJ) (2L \deg(f) \#(\CJ))^{n/L}. $$

We show in~\S\ref{ss:finiteness} that the pressure function~$\pressure$ of the canonical induced map associated to~$(\hV, V)$, defined in~\S\ref{ss:two variable pressure}, is finite on a neighborhood of~$(t, p) = (t_0, p_0)$.
In~\S\ref{ss:vanishing} we show that for each~$t$ close to~$t_0$ the function~$\pressure$ vanishes at~$(t, p) = (t, P(t))$.
Then Theorem~\ref{t:nice thermodynamics} follows from Theorem~\ref{t:lifting}.

\subsection{The function~$\pressure$ is finite on a neighborhood of $(t, p) = (t_0, P(t_0))$}\label{ss:finiteness}
By the considerations in~\S\ref{ss:two variable pressure}, to show that~$\pressure$ is finite on a neighborhood of~$(t, p) = (t_0, p_0)$ we just need to show that there are~$t < t_0$ and $p < P(t_0)$ such that
\begin{equation}\label{e:finiteness}
\sum_{W \in \fD} \exp( - p m_W) \diam(W)^t < + \infty.
\end{equation}

Let~$t < t_0$ and~$p < P(t_0)$ be given by part~1 of~Proposition~\ref{p:pull-back contribution}.
Taking~$t$ and~$p$ closer to~$t_0$ and~$P(t_0)$, respectively, we assume that there is~$\varepsilon > 0$ sufficiently small so that
$$ p - \max \{ -t \chiinf, - t \chisup \} - |t| \varepsilon > \tfrac{1}{2} (P(t_0) - \max \{ -t_0 \chiinf, - t_0 \chisup \}), $$
and put
\begin{equation*}\label{e:folga}
\gamma \= \exp( -p + \max \{ -t \chiinf, - t \chisup \} + |t| \varepsilon ) \in (0, \gamma_0).
\end{equation*}

For each $c \in \CJ$ we have, by applying part~2 of Proposition~\ref{p:pull-back contribution} to~$\badp = \hV^c$,
\begin{equation}\label{e:first return}
\sum_{W \in \fD_{\hV^c}} \exp( - p m_W) \diam(W)^t \le C_1 (\deg(f) + 1).
\end{equation}
Since for each pull-back~$\badp$ of~$\hV$ we have~$\ell(\badp) \le n / L + 1$, using part~2 of Proposition~\ref{p:pull-back contribution} again and letting~$C_2 = C_1 (\deg(f) + 1) \#(\CJ)$ we obtain
\begin{multline*}
\sum_{\badp \text{ bad pull-back of~$\hV$}} \sum_{W \in \fD_{\badp}} \exp( - p m_W) \diam(W)^t
\\ \le
C_1 \sum_{\badp \text{ bad pull-back of~$\hV$}} (\deg(f) + 1)^{\ell(\badp)} \cdot \gamma^{m_{\badp}}
\\ \le
C_2 \sum_{n = 1}^{+\infty} \left( (2L \deg(f) (\deg(f) + 1) \#(\CJ))^{1/L} \cdot \gamma \right)^n.
\end{multline*}
As~$\gamma \in (0, \gamma_0)$, we have by~\eqref{e:constants choice} that the sum above is finite.
Then~\eqref{e:finiteness} follows from~\eqref{e:first return} and Lemma~\ref{l:decomposition into bad}.
\subsection{For each~$t$ close to~$t_0$ we have $\pressure(t, P(t)) = 0$}\label{ss:vanishing}
Recall that for a given~$t_0 \in (\tneg, \tpos)$ we have fixed a sufficiently small pleasant couple~$(\hV, V)$, and that we denote by~$\pressure$ the corresponding pressure function defined in~\S\ref{ss:two variable pressure}.
Furthermore, in~\S\ref{ss:finiteness} we have shown that the function~$\pressure$ is finite on a neighborhood of~$(t_0, P(t_0))$.
We will show now that for~$t$ close to~$t_0$ the function~$\pressure$ vanishes at~$(t, P(t))$, thus completing the proof of Theorem~\ref{t:nice thermodynamics}.

In view of Lemma~\ref{l:pressure} we just need to show that for each~$t$ close to~$t_0$ we have $\pressure(t, P(t)) \ge 0$.
Suppose by contradiction that in each neighborhood of~$t_0$ we can find~$t$ such that $\pressure(t, P(t)) < 0$.
As~$\pressure$ is finite on a neighborhood of~$(t, p) = (t_0, P(t_0))$, it follows that~$\pressure$ is continuous at this point (Lemma~\ref{l:pressure}).
Thus there are
$$ t \in (\tneg, \tpos)
\text{ and }
p \in (\max \{ - t \chiinf, -t \chisup \}, P(t)), $$
such that $\pressure(t, p) < 0$, and such that the conclusion of part~1 of Proposition~\ref{p:pull-back contribution} holds for these values of~$t$ and~$p$.
However, this contradicts following lemma.
\begin{lemm}\label{l:from induced to full tree}
Let~$t \in (\tneg, \tpos)$ and $p > \min \{ - t\chiinf, -t\chisup \}$ be such that $\pressure(t, p) < 0$ and such that the conclusion of part~1 of Proposition~\ref{p:pull-back contribution} holds for these values of~$t$ and~$p$.
Then $p \ge P(t)$.
\end{lemm}
\begin{proof}
Fix $z_0 \in V$ such that all,~\eqref{e:tree pressure}, \eqref{e:individual minimal pressure}, and \eqref{e:individual maximal pressure} hold.
To prove the lemma we just need to show that
$$ \sum_{n = 1}^{+ \infty} \exp( - p n) \sum_{y \in f^{-n}(z_0)} |(f^n)'(y)|^{-t}
< + \infty. $$

\partn{1}
Given an integer~$n \ge 1$ an element~$y \in f^{-n}(V)$ is a \emph{univalent iterated pre-image of order~$n$} if the pull-back of~$\hV$ by~$f^n$ containing~$y$ is univalent.
Recall that for an integer~$n \ge 1$ an element~$y$ of $f^{-n}(V)$ is a bad iterated pre-image of~$z_0$ of order~$n$ if for every~$j \in \{ 1, \ldots, n \}$ such that~$f^j(y) \in V$ the pull-back of~$\hV$ by~$f^n$ containing~$y$ is not univalent.

For~$y \in f^{-n}(z_0)$ there are three cases: $y$ is univalent, bad, or there is $m \in \{1, \ldots, n - 1 \}$ such that~$f^m(y) \in V$, such that~$f^m(y)$ is a bad iterated pre-image of~$z_0$ of order $n - m$ and such that~$y$ is a univalent iterated pre-image of~$f^m(y)$ of order~$m$.
In fact, if~$y \in f^{-n}(z_0)$ is not bad, then there is~$m \in \{1, \ldots, n - 1 \}$ such that~$f^m(y) \in V$ and such that~$y$ is a univalent iterated pre-image of~$f^m(y)$.
If~$m$ is the largest integer with this property, then there are two cases.
Either $m = n$ and then~$y$ is a univalent iterated pre-image of~$z_0$, or $m < n$ and then~$f^m(y)$ is a bad iterated pre-image of~$z_0$.

Therefore, if for each $w \in V$ we put
$$
U(w) \= 1 + \sum_{n = 1}^{+ \infty} \exp(-pn) \sum_{y \in f^{-n}(w), \text{ univalent}} |(f^n)'(y)|^{- t},
$$
then we have
\begin{multline}\label{e:tree from univalent}
1 + \sum_{n = 1}^{+ \infty} \exp( - p n) \sum_{y \in f^{-n}(z_0)} |(f^n)'(y)|^{-t}
\\
= U(z_0) + \sum_{n = 1}^{+ \infty} \exp(- p n) \sum_{w \in f^{-n}(z_0), \text{ bad}} |(f^n)'(w)|^{-t} U(w).
\end{multline}

As $p > \max \{ - t \chiinf, - t \chisup \}$, by Lemma~\ref{l:bad pull-backs} and~\eqref{e:constants choice} it follows that
$$
\sum_{n = 1}^{+ \infty} \exp( - p n) \sum_{w \in f^{-n}(z_0), \text{ bad}}   |(f^n)'(w)|^{-t}
<
+ \infty.
$$
So by~\eqref{e:tree from univalent}, to prove the lemma it is enough to prove that the supremum $\sup_{w \in V} U(p, w)$ is finite.

\partn{2}
Denote by~$L_V$ the first entry map to $V$, which is defined on the set of points $y \in \CC \setminus V$ having a good time, by $L_V (y) = f^{m(y)}(y)$.
Note that for each $w_0 \in V$, each integer $n \ge 1$ and each univalent iterated pre-image $y \in f^{-n}(w_0)$ of~$w_0$ of order~$n$, we have that $m(y) \le n$ and that $L_V(y) \in V$ is a univalent iterated pre-image of $w_0$ of order $n - m(y)$.
Moreover, note for each $k \ge 1$, each element of $F^{-k}(w_0)$ is a univalent iterated pre-image of $w_0$.
Conversely, for each univalent iterated pre-image $y$ of $w_0$ there is an integer $k \ge 1$ such that $F^k$ is defined at $y$ and $F^k(y) = w_0$ (see~\S\ref{ss:canonical induced map}).
Therefore, if for $z \in V$ we put
$$
L(z) \= 1 + \sum_{y \in L_V^{-1}(z_0)} \exp( - p m(y)) |(f^{m(y)})'(y)|^{-t},
$$
then we have,
\begin{equation}\label{e:univalent from landing}
U(w_0)
=
L(w_0) + \sum_{k = 1}^{+ \infty} \sum_{y \in F^{-k}(w_0)} \exp( - p m(y)) |(F^k)'(y)|^{-t} L(y).
\end{equation}

Since by hypothesis $\pressure(t, p) < 0$, for each $w \in V$ the double sum
\begin{equation*}
T_F(w) \= \sum_{k = 1}^{+ \infty} \sum_{y \in F^{-k}(w)} \exp( - p m(y) ) |(F^k)'(y)|^{-t},
\end{equation*}
is finite.

On the other hand, since the conclusion of part~1 of Proposition~\ref{p:pull-back contribution} holds, for each $z \in V$ the sum~$L(z)$ is finite.
By bounded distortion it follows that
$$
C' \= \sup_{z \in V} L(z) < + \infty.
$$

Thus, by~\eqref{e:univalent from landing} for each $w \in V$ we have $U(p, w) \le C' T_F(w) < + \infty$, and by bounded distortion $\sup_{w \in V} U(p, w) < + \infty$.
This completes the proof of the lemma.
\end{proof}

\appendix
\section{Puzzles and nice couples}\label{s:nice puzzles}
This appendix is devoted to showing that several classes of polynomials satisfy the conclusions of the Main Theorem.
In~\S\ref{ss:non-renormalizable} we consider the case of at most finitely renormalizable polynomials without indifferent periodic points, in~\S\ref{ss:infinitely renormalizable} we consider the case of some infinitely renormalizable quadratic polynomials, and finally in~\S\ref{ss:real quadratic} we consider the case of quadratic polynomials with real coefficients.

\subsection{At most finitely renormalizable polynomials }
\label{ss:non-renormalizable}
The purpose of this section is to prove the following result.
We thank Weixiao Shen for providing the main idea of the proof.
\begin{theoalph}\label{t:non-renormalizable}
Every at most finitely renormalizable complex polynomial without indifferent periodic points has arbitrarily small nice couples.
Furthermore, these nice couples can be formed by nice sets that are finite unions of puzzle pieces.
\end{theoalph}

See~\cite[Proposition~5]{CaiLi09} for a somewhat similar result in the case of multimodal maps.

The proof relies on the fundamental result that diameters of puzzles tend uniformly to~$0$ as their depth tends to $\infty$ \cite{KozvSt09}; see also ~\cite{QiuYin09} for the case when the Julia set is totally disconnected.

Let~$f$ be an at most finitely renormalizable polynomial, and consider the puzzle construction described in~\cite[\S2.1]{KozvSt09}.
Given an integer $n \ge 0$ we denote by~$\Upsilon_n$ the collection of all puzzles of depth~$n$, which are by definition open sets.
For~$P \in \Upsilon_n$ and $p \in P$ we put~$P_n(p) \= P$.
We will assume that every critical point of~$f$ in~$J(f)$ is contained in a puzzle piece.
It is always possible to do the puzzle construction with this property.
This follows from the fact that in each periodic connected component of~$J(f)$ that is not reduced to a single point, there are infinitely many separating periodic points, see for example~\cite[\S A.1]{LiShe0911}.
We remark that the main technical results of~\cite{KozvSt09}, including the ``complex \emph{a priori} bounds'', are stated for ``complex box mappings'', and they are thus independent of the periodic points used to construct the puzzle pieces.

For~$z \in \C$ put~$\cO_f(z) = \bigcup_{n \ge 1} f^n(z)$, and put
$$ \delta_0
\=
\min \left\{ \dist\left(c, \overline{\cO_f(c')}\right) \mid c, c' \in \CJ, c \not \in \overline{\cO_f(c')} \right\}. $$
Let~$n_0 \ge 1$ be a sufficiently large integer so that the diameter of each puzzle piece of depth~$n_0$ is strictly smaller than~$\delta_0/2$.
Furthermore, let~$n_1 > n_0$ be a sufficiently large integer such that for each~$c \in \CJ$ we have $\overline{P_{n_1}(c)} \subset P_{n_0}(c)$, and such that for distinct~$c, c' \in \CJ$ we have $\overline{P_{n_1}(c)} \cap \overline{P_{n_1}(c')} = \emptyset$.
The following is a straightforward consequence of our choice of~$n_1$.
\begin{lemm}\label{l:non-accumulating}
For each~$c, c' \in \CJ$ such that~$c' \not \in \overline{\cO_f(c)}$, and each~$n \ge 1, m \ge 1, n' \ge n_1$ we have
\begin{equation*}\label{e:non-accumulating}
f^m(\partial P_{n}(c)) \cap \overline{P_{n'}(c')} = \emptyset.
\end{equation*}
\end{lemm}

Given a subset~$\sC$ of~$\CJ$ and a function $\nu : \sC \to \N$ we put
$$ P(\nu) \= \bigcup_{c \in \sC} P_{\nu(c)}(c). $$
We will say that~$\nu$ is \emph{nice} (resp. \emph{strictly nice}) if for every $c \in \sC$ we have $\nu(c) \ge n_1$, and if for every integer~$m \ge 1$ we have
$$ f^m(\partial P(\nu)) \cap P(\nu) = \emptyset
\text{ (resp. $f^m(\partial P(\nu)) \cap \overline{P(\nu)} = \emptyset$)}. $$

The following lemma is part~$2$ of Lemma~$2.2$ of~\cite{KozvSt09}.
\begin{lemm}\label{l:recurrent}
For every recurrent critical point~$c$ in~$J(f)$ there is a strictly nice function defined on~$\{c \}$.
\end{lemm}
\begin{proof}
Let~$\ell_0 \ge 0$ be a sufficiently large integer so that~$\overline{P_{\ell_0}(c)} \subset P_0(c)$, and note that for every~$m \ge \ell_0$ the set~$f^m(\partial P_{\ell_0}(c))$ is disjoint from~$P_0(c)$ by the puzzle structure, and hence it is disjoint from~$\overline{P_{\ell_0}(c)}$.

Define inductively a strictly increasing sequence of integers~$( \ell_k )_{k \ge 1}$ as follows.
Suppose that for some~$k \ge 0$ the integer~$\ell_k$ is already defined.
Then we denote by~$m_k$ the least integer such that~$f^{m_k}(c) \in P_{\ell_k}(c)$, and put~$\ell_{k + 1} = \ell_k + m_k$.

Clearly the sequence~$(\ell_k)_{k \ge 0}$ is strictly increasing, so~$\diam(P_{\ell_k}(c)) \to 0$ as $k \to + \infty$, and therefore~$m_k \to + \infty$ as $k \to + \infty$.
Let~$k \ge 1$ be sufficiently large so that~$m_k \ge \ell_0$, and so that for every~$m \in \{ 1, \ldots, \ell_0 \}$ the sets~$f^m(P_{\ell_k}(c))$ and~$P_{\ell_k}(c)$ have disjoint closures.
We will show that for each~$m \ge 1$ the set~$f^m(\partial(P_{\ell_k}(c)))$ and~$\overline{P_{\ell_k}(c)}$ are disjoints, which shows that the function~$\nu : \{ c \} \to \N$ defined by~$\nu(c) = \ell_k$ is strictly nice.
Suppose by contradiction that for some~$m \ge 1$ the set~$f^m(\partial P_{\ell_k}(c))$ intersects~$\overline{P_{\ell_k}(c)}$.
This implies that~$f^{\ell_k - \ell_0 + m}(\partial P_{\ell_k}(c)) = f^m(\partial P_{\ell_0}(c))$ intersects~$f^{\ell_k - \ell_0}(\overline{P_{\ell_k}(c)}) = \overline{P_{\ell_0}(c)}$.
By our choice of~$k$ we have~$m \ge \ell_0$, so we get a contradiction with our choice of~$n$.
\end{proof}

For a strictly nice function~$\nu : \sC \to \N$, denote by~$D_\nu$ the set of those points~$z \in \C$ for which there is an integer~$m \ge 1$ such that~$f^m(z) \in P(\nu)$, and for each~$z \in D_\nu$ denote by~$m_\nu(z)$ the least such integer, and by~$c_\nu(z)$ the critical point~$c$ in~$\sC$ such that~$f^{m_\nu(z)}(z) \in P_{\nu(c)}(c)$.
Furthermore we denote by $\sE_\nu : D_\nu \to P(\nu)$ the map defined by~$\sE_\nu(z) \= f^{m_\nu(z)}(z)$.

For a subset~$\sC$ of~$\CJ$ we put
$$ \sN_{\sC} \= \{ c \in \sC \text{ such that } \overline{\cO_f(c)} \cap \sC = \emptyset \}. $$
For a strictly nice function~$\nu$ defined on~$\sC$ let~$\sR\nu : \sC \setminus \sN_{\sC} \to \N$ be the function defined by,
$$ \sR\nu(c) = \nu(c_\nu(c)) + m_\nu(c). $$
By definition, $P_{\sR\nu(c)}(c)$ is the pull-back of~$P_{\nu(c_\nu(c))}(c_\nu(c))$ by~$f^{m_\nu(c)}$ containing~$c$.

Now we shall prove the key technical lemma which provides the inductive step to construct nice couples starting from the existence of ``nice couples'' around a single recurrent critical point.
This procedure resembles the procedure to build `complex box mappings' in~\cite{KozvSt09} or $\tau$-nice sets in \cite[Proposition 5]{CaiLi09} (in the multimodal interval setting).
\begin{lemm}\label{l:nice functions}
For a subset~$\sC$ of~$\CJ$ the following properties hold.
\begin{itemize}
\item[1.]
Let $\nu : \sC \setminus \sN_{\sC} \to \N$ be a strictly nice function.
Then for each sufficiently large integer~$n$ the function $\tilde{\nu} : \sC \to \N$ defined by $\tilde{\nu}|_{\sC \setminus \sN_{\sC}} = \nu$ and $\nu^{-1}(n) = \sN_{\sC}$, is strictly nice.
\item[2.]
Let~$\nu : \sC \to \N$ be strictly nice.
Then for each sufficiently large integer~$n \ge n_1$ the function~$\nu' : \sC \to \N$ defined by~$(\nu')^{-1}(n) = \sN_{\sC}$, and $\nu'|_{\sC \setminus \sN_{\sC}} = \sR \nu$ is strictly nice, we have $\overline{P(\nu')} \subset P(\nu)$, and for each integer~$m \ge 1$ we have
$$ f^m(\partial P(\nu')) \cap P(\nu) = \emptyset. $$
\item[3.]
Let~$\tilde{c} \in \CJ$ not in~$\sC$, be such that $\overline{\cO_f(\tilde{c})} \cap \sC \neq \emptyset$.
If there is a strictly nice function defined on~$\sC$, then there is also one defined on~$\sC \cup \{ \tilde{c} \}$.
\end{itemize}
\end{lemm}
\begin{proof}
\

\partn{1}
Put~$k = \max \{ \nu(c) \mid c \in \sC \setminus \sN_{\sC} \} \ge n_1$, let~$n > k$ be a sufficiently large integer so that for all~$c \in \sN_{\sC}$ we have~$\overline{P_n(c)} \subset P_k(c)$, and consider the function~$\tilde{\nu}$ defined as in the statement of the lemma for this choice of~$n$.
Then for each $c \in \sC \setminus \sN_{\sC}$ and~$m \ge 1$ we have
$$ f^m(\partial P_{\tilde{\nu}(c)}(c)) \cap \overline{P(\tilde{\nu})} = \emptyset. $$
On the other hand, since~$n > k \ge n_1$, by~Lemma~\ref{l:non-accumulating} the same property holds for each~$c \in \sN_{\sC}$.

\partn{2}
Let~$n \ge n_1$ be a sufficiently large integer such that for all~$c \in \sN_{\sC}$ we have~$\overline{P_n(c)} \subset P_{\nu(c)}(c)$, and let~$\nu'$ be the function defined in the statement of the lemma for this choice of~$n$.

Since~$\nu$ is strictly nice we have~$\overline{P_{\sR \nu}(\sC \setminus \sN_{\sC})} \subset P_{\nu}(\sC)$.
So, by our choice of~$n$ we have~$\overline{P_{\nu'}(\sC)} \subset P_{\nu}(\sC)$.
On the other hand, by the definition of~$\sR \nu$ it follows that for each~$c \in \sC \setminus \sN_{\sC}$ and each~$m \ge 1$ the set $f^m(\partial P_{\sR\nu(c)}(c))$ is disjoint from~$P_\nu(\sC)$, and hence from~$\overline{P_{\nu'}(\sC)}$.
Finally, by Lemma~\ref{l:non-accumulating}, for each~$c \in \sN_{\sC}$ and each~$m \ge 1$ the set~$f^m(\partial P_n(c))$ is disjoint from~$P_{\nu}(\sC)$, and hence from~$\overline{P_{\nu'}(\sC)}$.

\partn{3}
Let~$\nu_0 : \sC \to \N$ be a strictly nice function.
By part~2 there is a sequence of strictly nice functions~$(\nu_k)_{k \ge 1}$ defined on~$\sC$, such that for each~$k \ge 1$ we have~$\nu_k|_{\sC \setminus \sN_{\sC}} = \sR \nu_{k - 1}$, and~$\overline{P(\nu_k)} \subset P(\nu_{k - 1})$.

Put~$L = \# \CJ$, and let~$\hat{c}$ be the critical point defined as follows.\footnote{The proof of this part is simpler in the case when the forward orbit of~$\tilde{c}$ is disjoint from $\Crit(f)$. We advise to restrict to this case on a first reading, taking $L = 0$ and~$\hat{c} = \tilde{c}$.}
If~$\sE_{\nu_L}(\tilde{c}) \not \in \sC$, then put~$\hat{c} \= \tilde{c}$, $\hat{v} \= \sE_{\nu_L}(\tilde{c})$, and~$\ell = 0$.
Otherwise we let~$\ell \in \{ 1, \ldots, L \}$ be the largest integer such that for all~$j \in \{ 0, \ldots, \ell - 1 \}$ we have
$$ \sE_{\nu_{L - j}} \circ \cdots \circ \sE_{\nu_{L}}(\tilde{c}) \in \sC, $$
and then put
$$ \hat{c} \= \sE_{\nu_{L - (\ell - 1)}} \circ \cdots \circ \sE_{\nu_{L}}(\tilde{c}),
\text{ and } \hat{v} \= \sE_{\nu_{L - \ell}}(\hat{c}). $$

By definition we have~$\hat{v} \in P(\nu_{L - \ell})$, but $\hat{v} \not \in \sC$.
Let~$k \ge L - \ell$ be the largest integer such that $\hat{v} \in P(\nu_k)$, and note that~$m_{\nu_k}(\hat{c}) = m_{\nu_{L - \ell}}(\hat{c})$, and~$c_{\nu_k}(\hat{c}) = c_{\nu_{L - \ell}}(\hat{c})$.

Put
$$ \hat{n} \= \nu_k(c_{\nu_{k}}(\hat{v})) + m_{\nu_{k}}(\hat{v}) + m_{\nu_{k}}(\hat{c}), $$
so that $P_{\hat{n}}(\hat{c})$ is the pull-back of~$P(\nu_k)$ by~$f^{m_{\nu_{k}}(\hat{v}) + m_{\nu_{k}}(\hat{c})}$ containing~$\hat{c}$.

\partn{3.1}
We will show now that for every~$m \ge 1$ the set~$f^m(\partial P_{\hat{n}}(\hat{c}))$ is disjoint from~$\overline{P(\nu_{k + 1})}$.
We will use several times the fact that~$\nu_k$ is strictly nice.
By definition of~$\hat{n}$ the image of~$P_{\hat{n}}(\hat{c})$ by~$f^{m_{\nu_{k}}(\hat{v}) + m_{\nu_{k}}(\hat{c})}$ is a connected component of~$P(\nu_{k})$.
So for each~$m \ge m_{\nu_{k}}(\hat{v}) + m_{\nu_{k}}(\hat{c})$ the set~$f^m(\partial P_{\hat{n}}(\hat{c}))$ is disjoint from~$P(\nu_k)$, and hence from~$\overline{P(\nu_{k + 1})}$.
Since~$m_{\nu_{k}}(\hat{v})$ is the first entry time of~$\hat{v}$ to~$P(\nu_{k})$, it follows that the same property holds for each $m \in \{ m_{\nu_{k}}(\hat{c}) + 1, \ldots, m_{\nu_{k}}(\hat{v}) + m_{\nu_{k}}(\hat{c}) - 1 \}$.
On the other hand, by definition of~$\hat{n}$ the set~$f^{m_{\nu_{k}}(\hat{c})}(P_{\hat{n}}(\hat{c}))$ is the pull-back of~$P(\nu_k)$ by~$f^{m_{\nu_{k}}(\hat{v})}$ containing~$\hat{v}$, and by definition of~$k$ we have~$\hat{v} \in P(\nu_k) \setminus P(\nu_{k + 1})$.
As~$\nu_k$ is strictly nice, it follows that the sets~$f^{m_{\nu_{k}}(\hat{c})}(P_{\hat{n}}(\hat{c}))$ and $P(\nu_{k + 1})$ have disjoint closures.
Finally, since~$m_{\nu_k}(\hat{c})$ is the first entry time of~$\hat{c}$ to~$P(\nu_{k})$, it follows that for all~$m \in \{ 1, \ldots, m_{\nu_k}(\hat{c}) - 1 \}$ the sets~$f^m(P_{\hat{n}}(\hat{c}))$ and~$P(\nu_{k + 1})$ have disjoint closures.

\partn{3.2}
We will show now that for each integer~$m \ge 1$ the set~$f^m(\partial P(\nu_{k + 1}))$ is disjoint from~$\overline{P_{\hat{n}}(\hat{c})}$.
Suppose by contradiction that there is an integer~$m \ge 1$ and~$c \in \sC$ such that~$f^m(\partial P_{\nu_{k + 1}(c)}(c))$ intersects~$\overline{P_{\hat{n}}(\hat{c})}$.
Then the set~$f^{m_{\nu_k}(\hat{c}) + m}(\partial P_{\nu_{k + 1}(c)}(c))$ intersects the closure of~$f^{m_{\nu_k}(\hat{c})}(P_{\hat{n}}(\hat{c}))$.
This last set is a first return domain of~$P(\nu_k)$, and it is thus compactly contained in the open set~$P(\nu_k)$, because~$\nu_k$ is strictly nice.
We conclude that the set~$f^{m_{\nu_k}(\hat{c}) + m}(\partial P_{\nu_{k + 1}(c)}(c))$ intersects intersects the open set~$P(\nu_k)$.
However, this contradicts the fact that~$P_{\nu_{k + 1}(c)}(c)$ is a first return domain of~$P(\nu_k)$.

\partn{3.3}
If~$\hat{c} = \tilde{c}$, then the properties shown in parts~$3.1$ and~$3.2$ imply that the function~$\tilde{\nu} : \sC \cup \{ \tilde{c} \} \to \N$ defined by~$\tilde{\nu}(\tilde{c}) = \hat{n}$ and~$\tilde{\nu}|_{\sC} = \nu_{k + 1}$ is strictly nice.

If~$\hat{c} \neq \tilde{c}$, then we let~$\tilde{m}$ be the integer such that~$f^{\tilde{m}}(\tilde{c}) = \sE_{\nu_{L - (\ell - 1)}} \circ \cdots \circ \sE_{\nu_L}(\tilde{c})$, so that~$P_{\hat{n} + \tilde{m}}(\tilde{c})$ is the pull-back of~$P_{\hat{n}}(\hat{c})$ by~$\sE_{\nu_{L - (\ell - 1)}} \circ \cdots \circ \sE_{\nu_L}$.
Then we define~$\tilde{\nu} : \sC \cup \{ \tilde{c} \} \to \N$ by~$\tilde{\nu}(\tilde{c}) = \hat{n} + \tilde{m}$, and~$\tilde{\nu}|_{\sC} = \nu_{k + \ell}$.
The properties shown in parts~$3.1$ and~$3.2$ imply that~$\tilde{\nu}$ is strictly nice.
\end{proof}

\begin{proof}[Proof of Theorem~\ref{t:non-renormalizable}]
In view of part~$2$ of Lemma~\ref{l:nice functions} it is enough to show that there is a strictly nice function defined in all of~$\CJ$.

Let us say a critical point~$c \in \CJ$ is \emph{corresponded} if for each~$c' \in \CJ$ such that~$c' \in \overline{\cO_f(c)}$ we have~$c \in \overline{\cO_f(c')}$.
Denote by~$\sC_0$ the set of corresponded critical points in~$J(f)$.
Note that for each critical point~$c$ in~$J(f)$ that is not corresponded, the set~$\overline{\cO_f(c)}$ intersects~$\sC_0$.
So, using part~$3$ of Lemma~\ref{l:nice functions} inductively, it follows that to show the existence of a strictly nice function defined in all of~$\CJ$ it is enough to show the existence of a strictly nice function defined on~$\sC_0$.

Let~$\sim$ be the relation on~$\sC_0$ defined by~$c \sim c'$ if~$c = c'$ or~$c \in \overline{\cO_f(c')}$.
It follows from the definition of~$\sC_0$ that~$\sim$ is an equivalence relation.
Let~$\sC_1$ be a subset of~$\sC_0$ containing a unique element in each equivalence class of~$\sim$.
Thus for each~$c \in \sC_0$ the set~$\overline{\cO_f(c)}$ intersects~$\sC_1$.
Using part~$3$ of Lemma~\ref{l:nice functions} inductively it follows that to show that there is a strictly nice function defined on~$\sC_0$, it is enough to show that there is one defined on~$\sC_1$.

By definition of~$\sC_1$ for each~$c \in \sC_1$ the set~$\overline{\cO_f(c)}$ is disjoint from~$\sC_1 \setminus \{ c \}$.
Thus the set~$\sN_{\sC}$ defined as in the statement of Lemma~\ref{l:nice functions} for~$\sC = \sC_1$, is equal to the set of those~$c \in \sC_1$ such that~$c \not \in \overline{\cO_f(c)}$.
Equivalently, $\sN_{\sC}$ is the set of those non-recurrent critical points in~$\sC_1$.
Thus, by part~$1$ of this lemma we just need to show that there is a strictly nice function defined on~$\sC_2 \= \sC_1 \setminus \sN_{\sC_1}$.

For each~$c \in \sC_2$ let~$\nu_c$ be a strictly nice function defined on~$\{c \}$, given by Lemma~\ref{l:recurrent}.
Let~$\nu : \sC_2 \to \N$ be defined for each~$c \in \sC_2$ by~$\nu(c) = \nu_c(c)$.
As for each~$c \in \sC_2$ we have~$\overline{\cO_f(c)} \cap (\sC_2 \setminus \{ c \} ) = \emptyset$, by Lemma~\ref{l:non-accumulating} the function~$\nu$ is strictly nice.
This completes the proof of the theorem.
\end{proof}

\subsection{Infinitely renormalizable quadratic maps}
\label{ss:infinitely renormalizable}
The purpose of this section is to show that each infinitely renormalizable polynomial or polynomial-like map whose small critical Julia sets converge to~$0$ satisfy the hypotheses of Theorem~\ref{t:nice thermodynamics}.
This includes the case of infinitely renormalizable quadratic maps with \emph{a priori} bounds; see~\cite{KahLyu08,McM94} and references therein for results on~\emph{a priori} bounds.

The \emph{post-critical set} of a rational map~$f$ is by definition
$$ P(f) \= \overline{\bigcup_{n = 1}^{+ \infty} f^n(\Crit(f))}. $$
If~$f$ is an infinitely renormalizable quadratic-like map, then~$P(f)$ does not contain pre-periodic pionts~\cite[Theorem~8.1]{McM94}.
\begin{lemm}\label{l:post-critically free}
Let~$f$ be a rational map and let~$V$ be a nice set for~$f$ such that~$\partial V$ is disjoint from the post-critical set of~$f$.
Then for every neighborhood~$\tV$ of~$\overline{V}$ there is~$\hV \subset \tV$ such that~$(\hV, V)$ is a pleasant couple.
\end{lemm}
\begin{proof}
We will assume that~$P(f)$ contains at least three points; otherwise~$f$ is conjugated to a power map~\cite[Theorem~3.4]{McM94} and then the assertion is  vacuously true.
We will denote by~$\dist_{\hyp}$ the hyperbolic distance on~$\CC \setminus P(f)$ and by~$\| f' \|$ the derivative of~$f$ with respect to it.
Then by Schwarz lemma we have~$\| f' \| \ge 1$ on $\CC \setminus f^{-1}(P(f))$ (\emph{cf.}, ~\cite[Theorem~3.5]{McM94}).
Furthermore, for~$z \in \CC \setminus P(f)$ and $r > 0$ we denote by~$B_{\hyp}(z, r)$ the ball corresponding to the hyperbolic metric on~$\CC \setminus P(f)$.

Let~$\varepsilon > 0$ be sufficiently small such that $B_{\hyp}(\partial V, 2\varepsilon) \subset \tV$ and put
$$ \hV \= V \cup B_{\hyp}(\partial V, \varepsilon). $$
By construction~$\hV$ is a neighborhood of~$\overline{V}$ in~$\CC$ and the set~$\hV \setminus V$ is disjoint from~$P(f)$.
So for each pull-back~$W$ of~$V$ the set~$\hW \setminus W$ is disjoint from~$\Crit(f)$.
We thus have~$\hW \cap \Crit(f) = \emptyset$ when $W \cap V = \emptyset$.
On the other hand, since~$\| f' \| \ge 1$ on~$\CC \setminus f^{-1}(P(f))$, when~$W \subset V$ we have
$$
\dist_{\hyp} (\partial \hW, \overline{V} \setminus P(f))
\le
\dist_{\hyp}(\partial \hW, \partial W)
\le
\dist_{\hyp} (\partial \hV, \partial V)
\le
\varepsilon. $$
Hence $\hW \subset \hV$.
This shows that~$(\hV, V)$ is a pleasant couple for~$f$.
\end{proof}

In what follows we shall use some terminology of~\cite{McM94} and~\cite[\S2.4, Appendix~A]{AviLyu08}.
\begin{prop}\label{p:infinitely renormalizable}
Let~$f$ be an infinitely renormalizable quadratic-like map for which the diameters of small critical Julia sets converge to~$0$.
Then~$f$ is expanding away from critical points and has arbitrarily small pleasant couples having property~(*).
In particular the conclusions of Theorem~\ref{t:nice thermodynamics} hold for~$f$.
\end{prop}
\begin{proof}
We will show that there are arbitrarily small puzzles containing the critical point whose boundaries are disjoint from the post-critical set.
Then Lemma~\ref{l:post-critically free} implies that there are arbitrarily small pleasant couples.
That each of these pleasant couples satisfies property~(*) is a repetition of the proof of~\cite[Lemma~4.2.6]{MauUrb03}, using the fact that each puzzle is a quasi-disk and thus that it has the ``cone property'' of~\cite[\S4.2]{MauUrb03} with ``twisted angles''.

Let~$\cS\cR(f)$ be the set of all integers~$n \ge 2$ such that~$f^n$ is simply renormalizable and let~$J_n$ be the corresponding critical small Julia set.
Then~$J_n$ is decreasing with~$n$.
For each~$k \ge 1$ we denote by~$m(k)$ the $k$-th element of~$\cS\cR(f)$.

We consider the usual puzzle construction with the~$\alpha$-fixed point of~$f$.
Then for each~$\ell \ge 1$ there is a puzzle of depth~$\ell$, that we denote by~$P_\ell$, whose closure contains~$J_{m(1)}$.
We have~$\bigcap_{\ell = 1}^{+ \infty} \overline{P_\ell} = J_{m(1)}$.
More generally, by induction it can be shown that if for a given~$s \ge 1$ we consider the puzzle construction with the $\alpha$-fixed points of the renormalizations of~$f^{m(1)}$, $f^{m(2)}$, \ldots, $f^{m(s)}$, then for each~$\ell \ge 1$ there is a puzzle of depth~$\ell$ that contains~$J_{m(s)}$.
We will denote it by~$P_{s, \ell}$.
Thus~$P_{s, \ell}$ is bounded by a finite number of arcs in an equipotential line and by the closure of some pre-images of external rays landing at the~$\alpha$-fixed points of the renormalizations of~$f^{m(1)}$, $f^{m(2)}$, \ldots, $f^{m(s)}$.
In particular the intersection of~$\partial P_{s, \ell}$ with the Julia set is a finite set of pre-periodic points and it is thus disjoint from~$P(f)$ by~\cite[Theorem~8.1]{McM94}.
Furthermore we have
$$ \bigcap_{\ell = 1}^{+ \infty} \overline{P_{s, \ell}} = J_{m(s)} $$
and hence
$$ \lim_{s \to + \infty} \lim_{\ell \to + \infty} \diam(P_{s, \ell}) = 0. $$
This completes the proof that~$f$ has arbitrarily small pleasant couples having property~(*).

To show that~$f$ is expanding away from critical points we just need to show that for each~$s \ge 1$ and~$\ell \ge 1$ the map~$f$ is uniformly expanding on~$K(P_{s, \ell}) \cap J_{m(s)}$.
As this set is compactly contained in~$\C \setminus P(f)$, it is enough to show that the derivative~$\| f' \|$ of~$f$ with respect of the hyperbolic metric on this set is strictly larger than~$1$ on~$\C \setminus f^{-1}(P(f))$.
Since~$f^{-1}(P(f))$ contains~$P(f)$ strictly, this is a consequence of Schwarz lemma.
\end{proof}
\subsection{Quadratic polynomials with real coefficients}
\label{ss:real quadratic}
In this section we show that each quadratic polynomial satisfies the conclusions of the Main Theorem.

If~$f$ is at most finitely renormalizable without indifferent periodic points, then by Theorem~\ref{t:non-renormalizable} the map~$f$ satisfies the hypotheses of the Main Theorem.
If~$f$ is infinitely renormalizable, then it has \emph{a priori} bounds by~\cite{McM94}, so the diameters of the small Julia set converge to~$0$ and then the assertion follows from Proposition~\ref{p:infinitely renormalizable}.
See also~Remark~\ref{r:primitive}.

It remains to consider the case when~$f$ has an indifferent periodic point.
Fix~$t_0 \in (\tneg, \tpos)$.
Since~$f$ has real coefficients it follows that~$f$ has a parabolic periodic point, and since~$f$ is quadratic it follows that~$f$ does not have critical points in the Julia set.
Therefore the function~$\ln|f'|$ is bounded and H{\"o}lder continuous on~$J(f)$, and since the measure theoretic entropy of~$f$ is upper semi\nobreakdash-continuous~\cite{FreLopMan83,Lju83}, there is an equilibrium state~$\rho$ of~$f$ for the potential~$-t_0 \ln |f'|$.
Since~$f$ has a parabolic periodic point it follows that~$\tpos$ is the first zero of~$P$, so we have~$P(t_0) > 0$ and therefore the Lyapunov exponent of~$\rho$ is strictly positive.
Since by~\cite[Theorem~A and Theorem~A.7]{PrzRivSmi04} there is a~$(t_0, P(t_0))$\nobreakdash-conformal measure of~$f$ (see also~\cite{Prz99}),~\cite[Theorem~8]{Dob0804} implies that~$\rho$ is in fact the unique equilibrium state of~$f$ for the potential~$- t_0 \ln |f'|$.
The analyticity of~$P$ at~$t = t_0$ is given by~\cite{MakSmi00} when~$t_0 < 0$ and when~$t_0 \ge 0$ the fact that~$P$ is analytic at $t = t_0$ can be shown in an analogous way as in~\cite{StrUrb03}, using and induced map defined with puzzles pieces.
\begin{rema}\label{r:primitive}
We will now explain why we have introduced pleasant couples to deal with infinitely renormalizable quadratic-like maps as in Proposition~\ref{p:infinitely renormalizable} and with quadratic polynomials with real coefficients in particular.
Following~\cite{McM94} we call a renormalization of a quadratic-like map \emph{primitive} if the corresponding small Julia sets are pairwise disjoint.
If the first renormalization of a quadratic-like map~$f$ is primitive, then the usual puzzle construction produces a puzzle piece~$P$ containing the small critical Julia set, in such a way that the first return puzzle~$P_0$  to~$P$ containing the critical point is compactly contained in~$P$.
These puzzle pieces form a nice couple~$(P, P_0)$ for~$f$.
Since the puzzle~$P$ can be made arbitrarily close to the small critical Julia set, a slightly more general argument shows that a map as in Proposition~\ref{p:infinitely renormalizable} having infinitely many primitive renormalizations admits arbitrarily small nice couples.
The Feigenbaum quadratic polynomial is an example of an infinitely renormalizable quadratic map having no primitive renormalization and it is possible to show that such it does not have arbitrarily small nice couples.
However, the Feigenbaum polynomial does have arbitrarily small pleasant couples by Proposition~\ref{p:infinitely renormalizable}.
\end{rema}
\section{Rigidity, multifractal analysis, and level-1 large deviations}\label{s:applications}
The purpose of this appendix is to prove that, apart from some well-known exceptional maps, the pressure function of each of the maps considered in this paper is strictly convex on~$(\tneg, \tpos)$.
We derive consequences for the dimension spectrum for Lyapunov exponents (\S\ref{ss:Lyapunov spectrum}) and for pointwise dimensions of the maximal entropy measure (\S\ref{ss:dimension spectrum}), as well as some level-1 large deviations results (\S\ref{ss:large deviations}).
See~\cite{Pes97,Mak98} for background in multifractal analysis, and~\cite{DemZei98} for background in large deviation theory.

In what follows by a \emph{power map} we mean a rational map~$P(z) \in \C(z)$ such that for some integer~$d$ we have~$P(z) = z^d$.
\begin{theoalph}\label{t:strict convexity}
Let~$f$ be a rational map satisfying the hypotheses of Theorem~\ref{t:nice thermodynamics}.
If~$f$ is not conjugated to a power, Chebyshev or Latt{\`e}s map, then for every~$t \in (\tneg, \tpos)$ we have $P''(t) > 0$.
In particular
\begin{equation*}
\infd \= \inf \{ - P'(t) \mid t \in (\tneg, \tpos) \}
<
\supd \= \sup \{ - P'(t) \mid t \in (\tneg, \tpos) \}.
\end{equation*}
\end{theoalph}
It is well known that for a power, Chebyshev or Latt{\`e}s map, $\tpos = + \infty$ and the pressure function~$P$ is affine on~$(\tneg, + \infty)$; in particular in this case we have $\infd = \supd$.
For a general rational map~$f$ and for $t_0 \in (\tneg, 0)$, a result analogous to Theorem~\ref{t:strict convexity} was shown by~Makarov and Smirnov in~\cite[\S3.8]{MakSmi00}.
\begin{proof}
Suppose that for some~$t_0 \in (\tneg, \tpos)$ we have $P''(t_0) = 0$.
Let~$(\hV, V)$ be a pleasant couple as in~\S\ref{s:proof of nice thermodynamics}, so that the corresponding pressure function~$\pressure$ is finite on a neighborhood of~$(t, p) = (t_0, P(t_0))$, and such that for each~$t \in \R$ close to~$t_0$ we have $\pressure(t, P(t)) = 0$, see~\S\ref{ss:two variable pressure} for the definition of~$\pressure$.
Then the implicit function theorem implies that the function,
\begin{multline*}
p_0(\tau)
\=
\pressure(t_0 + \tau, P(t_0) + \tau P'(t_0))
\\ =
P(F, - t_0 \ln |F'| - P(t_0) m - \tau (\ln |F'| + P'(t_0) m)),
\end{multline*}
defined for~$\tau \in \R$ in a neighborhood of~$t = 0$, satisfies $p_0''(0) = 0$.

Let~$\rho$ be the equilibrium measure of~$F$ for the potential~$- t_0 \ln |F'| - P(t_0) m$ and put $\psi = - \ln |F'| - P'(t_0) m$.
Since for each~$t$ close to~$t_0$ we have~$\pressure(t, P(t)) = 0$, the implicit function theorem gives $p_0'(0) = 0$.
Thus, by part~1 of Lemma~\ref{l:tail estimate} and~\cite[Proposition~2.6.13]{MauUrb03} we have
$$ \int \psi d\rho = \int - \ln |F'| - P'(t_0) m d \rho = p_0'(0) = 0, $$
see also Remark~\ref{r:using symbolic}.
On the other hand, by part~1 of Lemma~\ref{l:tail estimate} and~\cite[Proposition~2.6.14]{MauUrb03}
$$ 0 = p_0''(0) = \sum_{k = 0}^{+ \infty} \left( \int \psi \circ F^k \cdot \psi d\rho - \left(\int \psi d\rho \right)^2 \right), $$
is the asymptotic variance of~$\psi$ with respect to~$\rho$, see also Remark~\ref{r:using symbolic}.
By part~1 of Lemma~\ref{l:tail estimate} and~\cite[Lemma~4.8.8]{MauUrb03} it follows that there is a measurable function $u : J(F) \to \R$ such that $\psi = u \circ F - u$, see also Remark~\ref{r:using symbolic}.
Put
$$\tJ \= \{ z \in \CC \setminus K(V) \mid f^{m(z)}(z) \in J(F) \} $$
and extend~$u$ to a function defined on~$\tJ$, that for each $z \in \tJ \setminus J(F)$ it is given by,
$$ u(z) = u(f^{m(z)}(z)) - \sum_{j = 0}^{m(z) - 1} (- \ln |f'(f^j(z))| - P'(t_0)). $$
An argument similar to the construction of the conformal measure given in the proof of Proposition~\ref{p:lifting conformal}, shows that we have $\ln |f'| = - P'(t_0) + u \circ f - u$ on~$\tJ$; see also~\cite[Proposition~B.2]{PrzRiv07}.
By construction this last set has full measure with respect to the equilibrium state of~$f$ for the potential $- t_0 \ln |f'|$, \emph{cf.} \S\ref{ss:equilibrium}.
Thus, an argument similar to the proof of \cite[\S\S5--8]{Zdu90} (see also~\cite[\S3.8]{MakSmi00} or~\cite[Theorem~3.1]{May02}) implies that~$f$ is a power, Chebyshev or Latt{\`e}s map.
\end{proof}

\subsection{Dimension spectrum for Lyapunov exponents}\label{ss:Lyapunov spectrum}
Let~$f$ be a rational map of degree at least two.
For $z \in \CC$ we define
$$ \chi(z) = \lim_{n \to + \infty} \frac{1}{n} \ln |(f^n)'(z)|,
$$
whenever the limit exists; it is called the \emph{Lyapunov exponent of~$f$ at~$z$}.
The \emph{dimension spectrum for Lyapunov exponents} is the function $L : (0, + \infty) \to \R$ defined by, 
$$ L(\alpha) \= \HD ( \{ z \in J(f) \mid \chi(z) = \alpha \} ). $$

Following~\cite{MakSmi96} we will say that~$f$ is \emph{exceptional} if there is a finite subset~$\Sigma$ of~$\CC$ such that
\begin{equation}\label{e:exceptional}
f^{-1}(\Sigma) \setminus \Crit(f) = \Sigma,
\end{equation}
see also~\cite[\S1.3]{MakSmi00}.
A rational map~$f$ is exceptional if and only if~$\tneg > - \infty$.
Furthermore, in this case there is a set~$\Sigma_f$ containing at most four points such that~\eqref{e:exceptional} is satisfied with~$\Sigma = \Sigma_f$, and such that each finite set~$\Sigma$ satisfying~\eqref{e:exceptional} is contained in~$\Sigma_f$.
Power, Chebyshev and Latt{\`e}s maps are all exceptional.
See~\cite{MakSmi96} for other examples of exceptional rational maps.

It has been recently shown in~\cite[Theorem~2]{GelPrzRam0809} that if~$f$ is not exceptional, or if~$f$ is exceptional and $\Sigma_f \cap J(f) = \emptyset$, then for each~$\alpha \in (0, + \infty)$ we have
$$ L(\alpha) = \frac{1}{\alpha} \inf \{ P(t) + \alpha t \mid t \in \R \}. $$
Equivalently, the functions~$\alpha \mapsto - \alpha L(\alpha)$ and~$s \mapsto P(-s)$ form a Legendre pair.
Note that a Chebyshev or a Latt{\`e}s map~$f$ is exceptional and $\Sigma_f$ intersects~$J(f)$.

The following is a direct consequence of Theorem~\ref{t:strict convexity}.
\begin{coro}
Let~$f$ be a rational map satisfying the hypotheses of Theorem~\ref{t:nice thermodynamics}.
Suppose furthermore that~$f$ is not conjugated to a power map, and that either~$f$ is not exceptional, or that~$f$ is exceptional and $\Sigma_f$ is disjoint from~$J(f)$.
Then the dimension spectrum for Lyapunov exponents of~$f$ is real analytic on~$(\infd, \supd)$. 
\end{coro}
\subsection{Dimension spectrum for pointwise dimension}
\label{ss:dimension spectrum}
Let~$\rho_0$ be the unique measure of maximal entropy of~$f$.
Then for $z \in J(f)$ we define
$$ \alpha(z) \= \lim_{r \to 0^+} \frac{\ln \rho_0(B(z, r))}{\ln r}, $$
whenever the limit exists; it is called the \emph{pointwise dimension of~$\rho_0$ at~$z$}.
The \emph{dimension spectrum for pointwise dimensions} is defined as the function
$$ D(\alpha) \= \HD ( \{ z \in J(f) \mid \alpha(z) = \alpha \}). $$

When~$f$ is a polynomial with connected Julia set we have~$P'(0) = - \ln \deg(f)$, so by~\cite[\S5]{MakSmi00} it follows that for~$\alpha \le 1$ we have,
$$ D(\alpha) = \inf \left\{ t + \alpha \frac{P(t)}{\ln \deg(f)} \mid t \le 0 \right\}. $$
Equivalently, the function~$\beta \mapsto - \beta D(\frac{1}{\beta})$ on $\beta \ge 1$ and the function $s \mapsto (\ln \deg(f))^{-1} P( - s)$ on $s \ge 0$ form a Legendre pair.
So the following is a direct consequence of Theorem~\ref{t:nice thermodynamics} and Theorem~\ref{t:strict convexity}.
\begin{coro}
Let~$f$ be a polynomial with connected Julia set satisfying the hypotheses of Theorem~\ref{t:nice thermodynamics}.
If~$f$ is not a power or Chebyshev map, then the dimension spectrum for pointwise dimensions of the maximal entropy measure of~$f$ is real analytic on~$\alpha < 1$.
\end{coro}
\begin{rema}
In the uniformly hyperbolic case one has
\begin{equation}\label{e:dimension spectrum}
D(\alpha) = L(\ln \deg(f)/\alpha).
\end{equation}
This also holds when the set of those~$z \in J(f)$ for which~$\chi(z)$ exists and satisfies~$\chi(z) \le 0$ has Hausdorff dimension equal to~$0$, like for rational maps satisfying the \TCEC~\cite[\S1.4]{PrzRiv07}.
In fact, it is easy to see that for~$z \in J(f)$ belonging to the ``conical Julia set'' and for which both~$\alpha(z)$ and~$\chi(z)$ exists, and~$\chi(z) > 0$, we have~$\alpha(z) = \ln \deg(f) / \chi(z)$.
Then~\eqref{e:dimension spectrum} follows from~\cite[Proposition~3]{GelPrzRam0809}, that the set of those~$z \in J(f)$ that are not in the conical Julia set and~$\chi(z) > 0$ has Hausdorff dimension equal to~$0$.
\end{rema}
\subsection{Large deviations}
\label{ss:large deviations}
The purpose of this section is to present a sample application of Theorem~\ref{t:strict convexity} to level\nobreakdash-1 large deviations, using the characterizations of the pressure function given in~\cite{PrzRivSmi04}.
See~\cite{ComRiv0812} and references therein for some level\nobreakdash-2 large deviation principles for rational maps.
\begin{coro}
Let~$f$ be a rational map satisfying the hypotheses of Theorem~\ref{t:nice thermodynamics}, and that is not conjugated to a power, Chebyshev, or Latt{\`e}s map.
Fix~$t_0 \in (\tneg, \tpos)$ and let~$\rho_{t_0}$ be the equilibrium state of~$f$ for the potential~$- t_0 \ln |f'|$.
Fix~$x_0 \in J(f)$ such that~\eqref{e:tree pressure} holds, and for each~$n \ge 1$ put
$$ \omega_n \= \sum_{x \in f^{-n}(x_0)} \frac{ |(f^n)'(x)|^{- t_0} }{\sum_{y \in f^{-n}(x_0)} |(f^n)'(y)|^{- t_0}} \delta_x. $$
Given~$\varepsilon \in (0, - P'(t_0) - \infd)$, let~$t(\varepsilon) \in (\tneg, t_0)$ be determined by~$P'(t(\varepsilon)) = P'(t_0) - \varepsilon$.
Then we have,
\begin{multline*}
\lim_{n \to + \infty} \frac{1}{n} \ln \omega_n \left\{ x \in J(f) \mid \frac{1}{n} \ln |(f^j)'(x)| > \int \ln |f'| d \rho_{t_0} + \varepsilon \right\}
\\ =
P(t(\varepsilon)) - P(t_0) - (t(\varepsilon) - t_0)P'(t(\varepsilon))
< 0.
\end{multline*}
Similarly, given~$\varepsilon \in (0, \supd + P'(t_0))$ let~$\tilde{t}(\varepsilon) \in (t_0, \tpos)$ be determined by $P'(\tilde{t}(\varepsilon)) = P'(t_0) + \varepsilon$.
Then we have,
\begin{multline*}
\lim_{n \to + \infty} \frac{1}{n} \ln \omega_n \left\{ x \in J(f) \mid \frac{1}{n} \ln |(f^j)'(x)| < \int \ln |f'| d \rho_{t_0} - \varepsilon \right\}
\\ =
P(\tilde{t}(\varepsilon)) - P(t_0) - (\tilde{t}(\varepsilon) - t_0) P'(\tilde{t}(\varepsilon)) < 0.
\end{multline*}
\end{coro}
For a rational map satisfying the TCE condition, or the weaker ``Hypothesis~H'' of~\cite{PrzRivSmi04}, a similar result can be obtained for periodic points.
See~\cite{Com09} and references therein for analogous statements in the case of uniformly hyperbolic rational maps, and~\cite{KelNow92} for similar results in the case of Collet-Eckmann unimodal maps and~$t_0$ near~$1$.
\begin{proof}
First observe that by the choice of~$x_0$, for each~$s \in \R$ we have
\begin{multline*}
\lim_{n \to + \infty} \frac{1}{n} \ln \int \exp(s \ln |(f^n)'|) d\omega_n
=
\lim_{n \to + \infty} \frac{1}{n} \ln \frac{\sum_{x \in f^{-n}(x_0)} |(f^n)'(x)|^{ - t_0 + s}}{\sum_{y \in f^{-n}(x_0)} |(f^n)'(y)|^{- t_0}}
\\ =
P(t_0 - s) - P(t_0).
\end{multline*}
We will apply the theorem in page~$343$ of~\cite{PlaSte75} to the space $J \= \prod_{n = 1}^{+ \infty} J(f)$ endowed with the probability measure $\mathsf{P} \= \bigotimes_{n = 1}^{+ \infty} \omega_n$.
Furthermore for each~$n \ge 1$ we take the random variable~$W_n : J \to \R$ as $W_n( \prod_{j = 1}^{+ \infty} x_j) \= \ln |(f^n)'(x_n)|. $
So for each~$s \in \R$ we have
$$ \int \exp(s W_n) d\mathsf{P} = \int \exp(s \ln |(f^n)'|) d\omega_n, $$
and by the computation above,
$$ \lim_{n \to + \infty} \frac{1}{n} \ln \int \exp(s W_n) d\mathsf{P}
=
P(t_0 - s) - P(t_0). $$
Using that $\int \ln |(f^n)'| d \rho_{t_0} = - P'(t_0)$ and that the function~$s \mapsto P(t_0 - s) - P(t_0)$ is real analytic and strictly convex on~$(t_0 - \tpos, t_0 - \tneg)$ by Theorem~\ref{t:strict convexity}, we obtain by the theorem in page~$343$ of~\cite{PlaSte75} that
\begin{multline*}
\lim_{n \to + \infty} \frac{1}{n} \ln \omega_n \left\{ \tfrac{1}{n} \ln |(f^n)'|
> \int \ln |(f^n)'| d \rho_{t_0} + \varepsilon \right\}
\\ =
P(t(\varepsilon)) - P(t_0) - (t(\varepsilon) - t_0)P'(t(\varepsilon)).
\end{multline*}

The second assertion is obtained analogously with~$W_n$ replaced by the function $ \prod_{j = 1}^{+ \infty} x_j \mapsto - \ln |(f^n)'(x_n)|$.
\end{proof}
\bibliographystyle{alpha}

\end{document}

%% file: bad.pstex_t
\begin{picture}(0,0)%
\includegraphics{bad.pstex}%
\end{picture}%
\setlength{\unitlength}{4144sp}%
\begingroup\makeatletter\ifx\SetFigFont\undefined%
\gdef\SetFigFont#1#2#3#4#5{%
  \reset@font\fontsize{#1}{#2pt}%
  \fontfamily{#3}\fontseries{#4}\fontshape{#5}%
  \selectfont}%
\fi\endgroup%
\begin{picture}(6012,3421)(831,-3381)
\put(2637,-994){\makebox(0,0)[lb]{\smash{{\SetFigFont{12}{14.4}{\familydefault}{\mddefault}{\updefault}{\color[rgb]{0,0,0}${\mathfrak L}_V$}%
}}}}
\put(5785,-1428){\makebox(0,0)[lb]{\smash{{\SetFigFont{12}{14.4}{\familydefault}{\mddefault}{\updefault}{\color[rgb]{0,0,0}$X$}%
}}}}
\put(5041,-2034){\makebox(0,0)[lb]{\smash{{\SetFigFont{12}{14.4}{\familydefault}{\mddefault}{\updefault}{\color[rgb]{0,0,0}$Y$}%
}}}}
\put(6429,-2351){\makebox(0,0)[lb]{\smash{{\SetFigFont{12}{14.4}{\familydefault}{\mddefault}{\updefault}{\color[rgb]{0,0,0}$f^{m_Y}$}%
}}}}
\put(5320,-913){\makebox(0,0)[lb]{\smash{{\SetFigFont{12}{14.4}{\familydefault}{\mddefault}{\updefault}{\color[rgb]{0,0,0}$V^c$}%
}}}}
\put(4375,-774){\makebox(0,0)[lb]{\smash{{\SetFigFont{12}{14.4}{\familydefault}{\mddefault}{\updefault}{\color[rgb]{0,0,0}${\mathfrak D}_Y$}%
}}}}
\put(1469,-1902){\makebox(0,0)[lb]{\smash{{\SetFigFont{12}{14.4}{\familydefault}{\mddefault}{\updefault}{\color[rgb]{0,0,0}${\mathfrak D}_Y$}%
}}}}
\put(2120,-1111){\makebox(0,0)[lb]{\smash{{\SetFigFont{12}{14.4}{\familydefault}{\mddefault}{\updefault}{\color[rgb]{0,0,0}$f$}%
}}}}
\put(4836,-2807){\makebox(0,0)[lb]{\smash{{\SetFigFont{12}{14.4}{\familydefault}{\mddefault}{\updefault}{\color[rgb]{0,0,0}$f^{m_Y}$}%
}}}}
\put(5776,-425){\makebox(0,0)[lb]{\smash{{\SetFigFont{12}{14.4}{\familydefault}{\mddefault}{\updefault}{\color[rgb]{0,0,0}$\widehat V^c$}%
}}}}
\put(1697,-1607){\makebox(0,0)[lb]{\smash{{\SetFigFont{12}{14.4}{\familydefault}{\mddefault}{\updefault}{\color[rgb]{0,0,0}$c$}%
}}}}
\put(1486,-3303){\makebox(0,0)[lb]{\smash{{\SetFigFont{12}{14.4}{\familydefault}{\mddefault}{\updefault}{\color[rgb]{0,0,0}$Y = \widehat V^c$}%
}}}}
\put(5196,-3303){\makebox(0,0)[lb]{\smash{{\SetFigFont{12}{14.4}{\familydefault}{\mddefault}{\updefault}{\color[rgb]{0,0,0}$f^{m_Y}(Y) = \widehat V^c$}%
}}}}
\end{picture}%